\documentclass[reqno]{amsart}
\usepackage{amssymb,amsmath,amsthm}
\usepackage{amsmath,amssymb,amsthm}
\usepackage{enumerate}
\usepackage{doi}
\usepackage{tikz}
\makeatletter
\@namedef{subjclassname@2020}{\textup{2020} Mathematics Subject Classification}
\makeatother

\theoremstyle{plain}
\newtheorem{theorem}{Theorem}[section]
\newtheorem{lemma}[theorem]{Lemma}
\newtheorem{corollary}[theorem]{Corollary}
\newtheorem{proposition}[theorem]{Proposition}

\theoremstyle{definition}

\newtheorem{remark}[theorem]{Remark}

\newcommand{\bt}{\mathbf{t}}
\newcommand{\bu}{\mathbf{u}}
\newcommand{\bv}{\mathbf{v}}
\newcommand{\bw}{\mathbf{w}}

\newcommand{\B}{\mathcal B}
\newcommand{\V}{\mathcal V}
\newcommand{\U}{\mathcal U}
\newcommand{\W}{\mathcal W}
\newcommand{\Lat}{\mathcal L}

\newcommand{\Var}{\mathsf{V}}
\newcommand{\Mod}{\mathsf{Mod}}

\newcommand{\htp}{\operatorname{ht}}

\begin{document}
\title[Equational Theories of Interval Semirings of Posets]
{Equational Theories of Interval Semirings of Posets}

\author{Zidong Gao}
\address{School of Mathematics, Northwest University, Xi'an, 710127, Shaanxi, P.R. China}
\email{zidonggao@yeah.net}

%
%

\author{Yilin Zhou}
\address{School of Mathematics, Northwest University, Xi'an, 710127, Shaanxi, P.R. China}
\email{zhou990619@163.com}


\begin{abstract}
We study the equational theory and subvariety structure of the
ai-semiring variety
$\V_\infty$ generated by all flat semirings $S(a_1\cdots a_k)$,
where the letters \(a_i\) are pairwise distinct.
Using interval semirings of posets, we characterize its
subdirectly irreducible members and describe variety membership
in terms of jointly separating families of strict order-preserving
maps. We obtain explicit finite identity bases for \(\V_\infty\)
and each $\V_k$ generated by $S(a_1\cdots a_k)$.
Consequently, every flat semiring \(S(W)\) associated with a nonempty set \(W\) of linear words is finitely based.
 This yields finitely based ai-semirings with exactly \(k\)-nilpotent multiplicative reduct for each \(k\geq 1\).

For each \(k\geq 1\), let \(\B_k\) be the subvariety of
\(\V_\infty\) defined by the \((k+1)\)-nilpotent identity.
We prove that each \(\B_k\) is generated by a finite interval
semiring and that every proper subvariety of \(\V_\infty\)
is contained in some \(\B_k\).
The variety \(\B_3\) is a Cross variety with exactly \(11\)
subvarieties, whereas \([\V_k,\B_k]\), \([\V_k,\V_{k+1}]\), and
\([\B_{k-1},\B_k]\) each contain continuum many
subvarieties for every $k\geq 4$. In particular, this provides infinitely many
finitely based finite semirings $S(a_1\cdots a_k)$
whose generated variety has continuum many subvarieties
for every \(k\geq 5\).

Let \(E_k\) be the bounded poset formed by two \(k\)-step chains sharing
only their least and greatest elements.
For \(k\geq4\), we prove that the interval semiring \(I(E_k)\) is nonfinitely based
and that \(\V_k\) is the largest proper subvariety of
\(\mathsf{V}(I(E_k))\).
Moreover, \(\mathsf{V}(I(E_4))\) is the unique limit subvariety
of \(\V_\infty\) and has exactly \(13\) subvarieties.
Finally, we exhibit two finite ai-semirings generating Cross
varieties whose direct product generates a variety with
continuum many subvarieties.
\end{abstract}

\maketitle

\section{Introduction}\label{sec:intro}
An \emph{additively idempotent semiring} (or ai-semiring for short)
is an algebra $(S, +, \cdot)$ of type $(2,2)$ such that the additive reduct $(S, +)$ is a commutative idempotent semigroup,
the multiplicative reduct $(S, \cdot)$ is a semigroup, and the distributive laws
\[
x(y+z) \approx xy + xz, \quad (x+y)z \approx xz + yz
\]
hold.
The additive reduct of  $S$ is a semilattice, and the natural partial order $\leq$ on $S$ defined by
\[
a \leq b\;\Leftrightarrow\; a+b=b
\]
is compatible with addition and multiplication.
Consequently, an ai-semiring is often called a \emph{semilattice-ordered semigroup} (see~\cite{kp}).
Unless otherwise stated, any order-theoretic statement concerning an ai-semiring refers to this order.

The class of ai-semirings
includes the Kleene semiring of regular languages~\cite{con}, the max-plus algebra~\cite{aei},
the semiring of all binary relations on a set~\cite{dol09}, the matrix semiring over an ai-semiring~\cite{bg},
and distributive lattices~\cite{bs}.
These and other similar algebras have played important roles
in several branches of mathematics, including algebraic geometry~\cite{cc}, tropical geometry~\cite{ms}, information science~\cite{gl}, and theoretical computer science~\cite{go}.

The finite basis problem for a class of ai-semirings, one of the most important problems in universal algebra,
concerns the classification of its members according to whether they are finitely based.
In the last two decades, the finite basis problem for ai-semirings has been intensively studied,
and considerable progress has been made;
see, for example,~\cite{dol07, dol09, dgv25, gjrz, gpz05, jrz, pas05, rlzc, rlyc, rjzl, sr, shap23, vol21, wrz, wzr, yrzs, zrc}.
For more information on the finite basis problem of other related algebras such as semigroups and monoids,
see the monograph~\cite{lee23}.

Dolinka~\cite{dol07} found the first example of a nonfinitely based finite ai-semiring.
Jackson~\cite{jac:flat} solved the finite basis problem for the flat extensions of finite groups,
thereby providing infinitely many minimal nonfinitely based finite ai-semirings.
Shao and Ren~\cite{sr} proved that every ai-semiring in the variety generated by all ai-semirings of order two is finitely based.
Jackson et al.~\cite{jrz} and Zhao et al.~\cite{zrc} classified three-element ai-semirings with respect to the finite basis property,
showing that $S_7$ is the unique nonfinitely based three-element ai-semiring.
The Cayley tables of $S_7$ are given in Table~\ref{tb24111401}.

\begin{table}[ht]
\caption{The Cayley tables of $S_7$} \label{tb24111401}
\begin{tabular}{c|ccc}
$+$      &$0$&$a$&$1$\\
\hline
$0$ &$0$&$0$&$0$\\
$a$      &$0$&$a$&$0$\\
$1$      &$0$&$0$&$1$\\
\end{tabular}\qquad
\begin{tabular}{c|ccc}
$\cdot$  &$0$&$a$&$1$\\
\hline
$0$ &$0$&$0$&$0$\\
$a$      &$0$&$0$&$a$\\
$1$      &$0$&$a$&$1$\\
\end{tabular}
\end{table}

Moreover, Volkov~\cite{vol21} and Jackson et al.~\cite{jrz} independently
resolved the finite basis problem for the ai-semiring whose multiplicative reduct is the six-element Brandt monoid.
Ren et al.~\cite{rjzl} provided an infinite number of minimal nonfinitely based ai-semiring varieties.

In much of the above work, flat semirings have played an important (sometimes even a decisive) role.
We now recall their definition.
By a \emph{flat semiring} we mean a nontrivial ai-semiring
whose multiplicative reduct has a zero element $0$ and satisfies $a+b=0$ for all distinct $a,b\in S$.
The three-element algebra $S_7$ is a typical example of a flat semiring.
As observed in~\cite[Lemma 2.2]{jrz} (see also~\cite[Lemma 4.1.1]{ek}),
a semigroup with zero \(0\) becomes a flat semiring if and only if it is \(0\)-cancellative; that is, for all $a, b, c \in S$,
\[
ab = ac \neq 0 \;\Rightarrow\; b = c \quad\text{and}\quad ab = cb \neq 0 \;\Rightarrow\; a = c.
\]

Let $\mathbf{F}$ denote the variety generated by all flat semirings. The following result, due to Jackson et al.~\cite[Lemma 2.1]{jrz}, solves the finite basis problem for $\mathbf{F}$ and characterizes its subdirectly irreducible members.

\begin{lemma}[{\cite[Lemma 2.1]{jrz}}]
\label{lem:flat}\label{siflat}\label{lem24121301}
Every subdirectly irreducible member of \(\mathbf F\) is flat.
Moreover, \(\mathbf F\) is defined, relative to the ai-semiring
axioms, by the finite family of identities
\begin{equation}\label{eq:flat}
 x_1(u+v)x_2+y_1uy_2
 \approx
 x_1(u+v)x_2+y_1vy_2,
\end{equation}
where any of \(x_1,x_2,y_1,y_2\) may be empty.
\end{lemma}

For any cancellative semigroup $(S, \cdot)$,
let $S^0$ denote the semigroup obtained from $S$ by adjoining a multiplicative zero $0$.
Then $S^0$ is $0$-cancellative, and therefore becomes a flat semiring;
we call it the \emph{flat extension} of $S$ and denote it by $\flat(S)$.

A nonempty subset \(J\) of an ai-semiring \(S\) is an \emph{ideal}
if it is a multiplicative ideal and an order filter for the
natural order. Equivalently,
\[
 SJ\cup JS\subseteq J,
 \qquad
 a\in J,\ a\leq b\ \Longrightarrow\ b\in J.
\]
The \emph{ideal quotient} \(S/J\) identifies all elements of \(J\)
and leaves every element of \(S\setminus J\) in a singleton class.
The class \(J\) absorbs both addition and multiplication in the
quotient. We will frequently use this construction to obtain flat semirings from subalgebras of direct products.

Next, we introduce another important class of flat semirings.
Let $W$ be a set of words in the free semigroup $X^+$ over $X$,
and let $W^{\leq}$ denote the set of all nonempty subwords of words in $W$.
Define a multiplication on $W^{\leq} \cup \{0\}$ by
\[
0 \cdot 0 = 0 \cdot \mathbf{u} = \mathbf{u} \cdot 0 = 0
\]
for all $\mathbf{u} \in W^{\leq}$, and for $\mathbf{u}, \mathbf{v} \in W^{\leq}$,
\[
\mathbf{u} \cdot \mathbf{v} =
\begin{cases}
\mathbf{u}\mathbf{v}, & \text{if } \mathbf{u}\mathbf{v} \in W^{\leq}, \\
0, & \text{otherwise}.
\end{cases}
\]
The resulting algebra is a semigroup and is denoted by $S(W)$.
If the empty word $1$ is allowed in the above construction,
then the corresponding semigroup becomes a monoid, which we denote by $M(W)$.
Working in the free commutative semigroup $X_c^+$ (or monoid $X_c^*$) yields the commutative analogues $S_c(W)$ and $M_c(W)$.
For $W = \{\mathbf{u}_i \mid i \in I\}$, we write $S(\mathbf{u}_i \mid i \in I)$ in place of $S(\{\mathbf{u}_i \mid i \in I\})$
to avoid nested braces, and similarly for $M(W)$, $S_c(W)$ and $M_c(W)$.

All of the above semigroups are $0$-cancellative,
and therefore become flat semirings (which we denote by the same symbols).
Moreover, these flat semirings can be realized as ideal quotients of flat extensions of free semigroups or free monoids,
as we now illustrate.
For example, the flat semiring \(S(W)\) is isomorphic to the ideal quotient of \(\flat(X^+)\) with respect to the ideal \(\flat(X^+)\backslash W^{\leq}\).

Gao, Jackson, Ren and Zhao \cite{gjrz} studied the variety
\(\mathsf{V}(S_7)\) and its subvariety \(\mathbf N\) defined by the
\(2\)-nil identity \(x^2\approx 0\) in detail. By introducing the
notion of a \emph{block hypergraph}, they completely characterized
the subdirectly irreducible members of \(\mathbf N\), and proved that
\(\mathbf N\) is exactly the variety generated by all flat semirings
\(S_c(a_1\cdots a_k)\).
For each \(k\geq 1\), let \(\mathbf N_k\) denote the subvariety of
\(\mathbf N\) determined by the \(k\)-nilpotent identity $x_1\cdots x_k\approx 0$.
Using the homomorphism theory of Kneser graphs, they proved that for
every \(k\geq 3\), the intervals
$[\mathsf{V}(S_c(a_1\cdots a_k)),\mathbf N_{k+1}]$ and
$[\mathbf N_k,\mathbf N_{k+1}]$
both contain continuum many subvarieties, and, using the notion of a
strong \(n\)-colouring, they proved that each \(\mathbf N_{k+1}\) is
not finitely generated. In particular, they showed that if a semiring
variety contains every flat semiring \(S_c(a_1\cdots a_k)\), then it
must contain continuum many subvarieties.

These results motivated the study of the three references
\cite{g,gr,gry} that follow the present line of investigation. In
\cite{gr}, the variety \(\mathbf{NF}_3\) generated by all
\(3\)-nilpotent flat semirings was systematically studied. By
introducing the notion of a \emph{graph semiring}, the subdirectly
irreducible members of \(\mathbf{NF}_3\) were completely characterized;
it was proved that \(\mathbf{NF}_3\) has continuum many subvarieties
and that every finite member of \(\mathbf{NF}_3\) is finitely based;
and finally it was shown that \(\mathbf{NF}_3\) has a unique limit
subvariety. Analogously to the idea of homomorphism independence
between hypergraphs, in \cite{gry} the authors used an infinite family of mutually free words
to prove that the interval
\([\mathsf{V}(S_7),\mathsf{V}(B_2^1)]\) contains continuum many subvarieties.
In \cite{g}, it was proved that if a semiring variety \(\mathcal V\) contains,
for every \(n\geq 1\), a semiring \(S_n\) whose multiplicative reduct is
the same as that of \(S_c(a_1\cdots a_n)\), then \(\mathcal V\) contains
continuum many subvarieties.  Using this, the first example of
a finitely based finite semiring with continuum many subvarieties was
obtained. Finally, the question was raised in \cite{g} whether a
noncommutative version of the result holds, that is, whether one can
characterize the subdirectly irreducible members, the cardinality of
the subvariety lattice, and the existence of limit subvarieties for the
variety \(\V_{\infty}\) generated by all \(S(a_1\cdots a_n)\). This is
precisely the starting point of the present paper. By introducing the
notion of interval semirings, we obtain a perfect analogue of the role
played by block hypergraph semirings in the variety \(\mathbf N\), and
derive some parallel conclusions as well as several other interesting
results.

The paper is organized as follows.
Section~\ref{sec:preliminaries} collects the required terminology and
facts about flat semirings and identities.
Section~\ref{sec:interval} introduces interval semirings and joint
separation.
Section~\ref{sec:bases} establishes the structural descriptions and
finite identity bases, and Section~\ref{sec:bounded} studies the
bounded-height varieties \(\B_k\).
Section~\ref{sec:crossb3} determines the subvariety lattice of
\(\B_3\), while Section~\ref{sec:independence} constructs the
continuum of subvarieties from height four onwards.
Section~\ref{sec:limit} proves the nonfinite basis results for
\(I(E_k)\) and identifies the limit variety
\(\mathsf{V}(I(E_4))\).
Section~\ref{sec:threeclasses} proves its uniqueness and gives the
direct-product application.

\section{Preliminaries}\label{sec:preliminaries}

We use standard terminology from universal algebra; see \cite{bs}.
\subsection*{Variety and equational basis}

A class of ai-semirings is a \emph{variety} if it is closed under
taking subalgebras, homomorphic images, and arbitrary direct products.
By Birkhoff's theorem, a class of ai-semirings is a variety if and only
if it is an \emph{equational class}; that is, the class of all
ai-semirings satisfying a certain set of identities. We write
\(\Mod(\Sigma)\) for the class of all ai-semirings satisfying a set
\(\Sigma\) of identities. Accordingly, every identity basis of a variety in this paper
is given relative to the ai-semiring axioms. The lattice of
subvarieties of a variety \(\mathcal{V}\) is denoted by
\(\Lat(\mathcal{V})\), and \(\mathcal{U}\vee\mathcal{W}\) denotes the
join of two varieties.

Let \(\mathcal{V}\) be an ai-semiring variety. Then \(\mathcal{V}\) is
\emph{finitely based} if it can be defined by a finite set of
identities; otherwise, it is \emph{nonfinitely based}.

For a class \(\mathcal{K}\) of ai-semirings, let
\(\mathsf{V}(\mathcal{K})\) denote the variety generated by
\(\mathcal{K}\); that is, the smallest variety containing
\(\mathcal{K}\). When \(\mathcal{K}=\{A_i\mid i\in I\}\), we usually
write \(\mathsf{V}(A_i\mid i\in I)\) instead of
\(\mathsf{V}(\{A_i\mid i\in I\})\) to avoid multiple parentheses. It is
well known that \(\mathsf{V}(\mathcal{K})\) consists precisely of all
homomorphic images of subalgebras of direct products of algebras in
\(\mathcal{K}\). An ai-semiring \(S\) is finitely based (respectively,
nonfinitely based) if the variety \(\mathsf{V}(S)\) is finitely based
(respectively, nonfinitely based).

\subsection*{Cross variety and limit variety}
A variety is \emph{finitely generated} if it is generated by a
finite algebra. It is \emph{locally finite} if each of its finitely
generated algebras is finite. A \emph{Cross variety} is a finitely
based, finitely generated variety with only finitely many subvarieties.
A \emph{limit variety} is a nonfinitely based variety all of whose
proper subvarieties are finitely based.

\subsection*{Subdirectly irreducible algebra}
An algebra is \emph{subdirectly irreducible} if it is nontrivial and its
congruence lattice has a least nontrivial congruence. Here we adopt the
congruence-based definition and exclude the trivial algebra from being
subdirectly irreducible.
Every variety is generated by its subdirectly
irreducible members. In particular, a locally finite variety is
generated by its finite subdirectly irreducible members, and the same
is true of each of its subvarieties.

\subsection*{Nilpotence and subdirect irreducibility}

For a semiring with multiplicative zero \(0\), we say that it is
\emph{nil} if every element has a power equal to \(0\), and
\emph{\(k\)-nil} if \(a^k=0\) for every element \(a\). It is
\emph{\(k\)-nilpotent} if every product of \(k\) elements is \(0\),
and \emph{nilpotent} if it is \(k\)-nilpotent for some \(k\geq 1\).
It is \emph{exactly \(k\)-nilpotent} if it is \(k\)-nilpotent but not
\((k-1)\)-nilpotent.

It is easy to verify that every nilpotent semiring is nil, but the
converse is not true in general.

A nonzero element \(\omega\) of a semiring with zero is an
\emph{annihilator} if \(\omega s=s\omega=0\) for every \(s\).
For a semigroup \(S\), the notation \(S^1\) means \(S\) with an
identity adjoined if necessary.

\begin{lemma}[{\cite[Propositions 1.4 and 1.5]{gr}}]\label{sinilflat}
Let \(S\) be a nil flat semiring. Then \(S\) is subdirectly
irreducible if and only if it has a unique annihilator \(\omega\)
and \(\omega\in S^1aS^1\) for every \(a\ne0\).
In this case, \(\{0,\omega\}\) is the least nonzero multiplicative
ideal of \(S\). If \(S\) is nilpotent, then it is subdirectly
irreducible if and only if it has a unique annihilator.
\end{lemma}

\subsection*{Terms and identities}

Each element of the free semigroups $X^+$ is called a \emph{word} or a \emph{monomial}, written as bold lowercase letters $\mathbf{u}, \mathbf{v}, \mathbf{w},\ldots$, while ordinary lowercase letters $x,y,z,\ldots$ stand for variables in $X$.

From \cite[Theorem 2.5]{kp} we know that
the set $P_f(X^+)$ of all nonempty finite subsets of $X^+$,
equipped with addition as set-theoretic union and multiplication as elementwise product,
is the free ai-semiring over $X$.
Consequently, an \emph{ai-semiring term} (or simply a \emph{term}) is precisely a nonempty finite subset of $X^+$.
In keeping with standard semiring notation,
we write the term $\{\mathbf{u}_1, \mathbf{u}_2, \dots, \mathbf{u}_n\}$
as the \emph{polynomial} $\mathbf{u}_1 + \mathbf{u}_2 + \cdots + \mathbf{u}_n$.

Now let $\mathbf{w}$ be a word, let $\bu=\mathbf{u}_1 + \mathbf{u}_2 + \cdots + \mathbf{u}_n$ a polynomial. Then
\begin{itemize}
\item  $c(\mathbf{w})$ denotes the \emph{content} of $\mathbf{w}$; that is, the set of all variables occurring in $\mathbf{w}$.

\item $\ell(\mathbf{w})$ denotes the \emph{length} of $\mathbf{w}$; that is, the number of variables occurring in $\mathbf{w}$ counting multiplicities.

\item $c(\mathbf{u}) = \bigcup_{1\leq i\leq n} c(\mathbf{u}_i)$.

\item $\bu(x_1,\ldots,x_n)$ denotes a polynomial with $c(\bu)\subseteq \{x_1, \ldots, x_n\}$.

\end{itemize}

An \emph{ai-semiring identity} (or simply an \emph{identity}) is
a formal expression of the form $\bu \approx \bv$, where $\bu$ and $\bv$ are polynomials.
We say that an ai-semiring $S$ \emph{satisfies} an identity $\bu \approx \bv$ over $X=\{x_1,\ldots,x_n\}$, denoted $S\vDash \bu\approx \bv$, if
\[
\bu(a_1, a_2, \dots, a_n) = \bv(a_1, a_2, \dots, a_n)
\]
for all $a_1, a_2, \dots, a_n \in S$.
Equivalently, $\varphi(\bu) = \varphi(\bv)$ for every semiring homomorphism $\varphi \colon P_f(X^+) \to S$.
Such a homomorphism is also called an \emph{assignment} (or \emph{evaluation}),
and is uniquely determined by the images of the elements of $X$;
for convenience, we may denote it simply by $\varphi \colon X \to S$, or even restrict to $\varphi\colon c(\bu)\cup c(\bv)\to S$.

If \(\mathbf u=\bu_1+\cdots+\bu_m\) and
\(\mathbf v=\bv_1+\cdots+\bv_n\), then the identity
\(\mathbf u\approx\mathbf v\) is equivalent, over ai-semirings,
to the family
\[
 \mathbf u\approx\mathbf u+\mathbf v_j\quad(1\leq j\leq n),
 \qquad
 \mathbf v\approx\mathbf v+\mathbf u_i\quad(1\leq i\leq m).
\]
Thus it suffices to consider identities
\(\mathbf u\approx\mathbf u+\mathbf q\), where \(\mathbf u\)
is a polynomial and \(\mathbf q\) is a word.

Since the semirings studied in this paper are all related to flat
semirings, they possess an absorbing element \(0\). For an ai-semiring
term \(\mathbf{u}\), we write
\[
\mathbf{u}\approx 0
\] for the following
identities
\[
\mathbf{u}x\approx x\mathbf{u}\approx \mathbf{u}+x\approx \mathbf{u},
\]
where \(x\notin c(\mathbf{u})\).

\begin{lemma}\label{lem:flatuuq}
A flat semiring \(S\) satisfies
\(\mathbf u\approx\mathbf u+\mathbf q\) if and only if
\(\varphi(\mathbf q)=\varphi(\mathbf u)\) for every assignment
\(\varphi\) of the variables in
\(c(\mathbf u)\cup c(\mathbf q)\) such that
\(\varphi(\mathbf u)\ne0\).
\end{lemma}

\begin{proof}
If \(\varphi(\mathbf u)=0\), the required equality holds because
\(0\) absorbs addition. Otherwise, the flat addition law gives
\(\varphi(\mathbf u)+\varphi(\mathbf q)=\varphi(\mathbf u)\)
exactly when \(\varphi(\mathbf q)=\varphi(\mathbf u)\).
\end{proof}

\section{Interval semirings of posets}\label{sec:interval}
Let \(\mathbf w\) be a word such that every non-letter subword of it
occurs exactly once. Then each such subword is determined by the pair
\((i,j)\) of the positions of its first and last letters within
\(\mathbf w\). This observation underlies the bracket
representation of the flat semiring \(S(\mathbf w)\), as developed
in~\cite{rjzl}.

We now analyse the special case in which \(\mathbf w\) is a linear word
\(a_1\cdots a_k\); that is, the letters \(a_i\) are pairwise distinct.
However, we have made a suitable modification:
in this paper we use edge labels rather than vertex labels.
Concretely, for the \(k\)-step chain
\[
  C_k:\;0< 1<2<\cdots< k,
\]
the covering pair \((i-1,i)\) represents the single-letter subword
\(a_i\), and more generally the interval \((i-1,j)\) represents the
subword \(a_i a_{i+1}\cdots a_j\). In particular, the whole interval
\((0,k)\) corresponds to the whole word \(a_1\cdots a_k\); see
Figure~\ref{fig:chain-k}. The nonzero products are precisely those of
the form
\[
  (i,j)\cdot(j,t)=(i,t),\quad 0\leq i<j<t\leq k.
\]
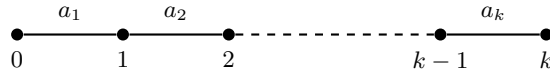
\begin{figure}[htbp]
\centering
\begin{tikzpicture}[
  scale=1,
  elem/.style={circle, fill=black, inner sep=1.6pt},
  lab/.style={font=\small},
  edge/.style={thick},
  dashedge/.style={thick, dashed}
]

\node[elem] (n0)  at (0,0) {};
\node[elem] (n1)  at (1.4,0) {};
\node[elem] (n2)  at (2.8,0) {};
\node[elem] (nk1) at (5.6,0) {};
\node[elem] (nk)  at (7.0,0) {};

\node[lab, below=3pt] at (n0)  {\(0\)};
\node[lab, below=3pt] at (n1)  {\(1\)};
\node[lab, below=3pt] at (n2)  {\(2\)};
\node[lab, below=3pt] at (nk1) {\(k-1\)};
\node[lab, below=3pt] at (nk)  {\(k\)};

\draw[edge] (n0) -- (n1);
\draw[edge] (n1) -- (n2);
\draw[edge] (nk1) -- (nk);

\draw[dashedge] (n2) -- (nk1);

\node[lab, above=2pt] at (0.7,0) {\(a_1\)};
\node[lab, above=2pt] at (2.1,0) {\(a_2\)};
\node[lab, above=2pt] at (6.3,0) {\(a_k\)};

\end{tikzpicture}
\caption{The edge-labelled chain \(C_k\).}
\label{fig:chain-k}
\end{figure}

Motivated by this, we now generalise the construction from chains to a
suitable class of posets. Unless explicitly specified otherwise, all
posets considered in this paper are assumed to be finite and to have no
isolated elements, that is, every element is comparable with at least
one other element.  For such a poset, it suffices to list its strict
cover relations, since the remaining order relations are obtained by
taking the reflexive and transitive closure. For example, the notation
\(C_k=\{0<1<2<\cdots<k\}\)
means that the ground set is \(\{0,1,\ldots,k\}\) and that the order
relation is the usual order of the integers.

For an arbitrary poset \(P\) (not necessarily finite), set
\[
I(P)=\{(p,q)\mid p,q\in P,\ p<q\}\dot\cup\{0\},
\]
where \(0\) is a fresh element. Define multiplication on \(I(P)\) by
\[
(p,q)(r,s)=
\begin{cases}
(p,s), & q=r,\\
0,     & q\ne r,
\end{cases}
\]
and
\[
0\cdot (p,q)=(p,q)\cdot 0=0\cdot 0=0.
\]
A direct check shows that \((I(P),\cdot)\) is a semigroup with zero,
and that it is \(0\)-cancellative. Hence, equipping \(I(P)\) with the
flat addition
\[
x+y=
\begin{cases}
x, & x=y,\\
0, & x\ne y,
\end{cases}
\]
makes \(I(P)\) into a flat semiring. We refer to \(I(P)\) as the
\emph{interval semiring of \(P\)}.

To describe the algebraic behaviour of \(I(P)\), we shall also need the
following terminology for posets. The \emph{height} of a poset \(P\) is
the maximum length of a chain in \(P\), denoted by
\(\htp(P)\). For \(a,b\in P\) with \(a\leq b\), let
\(\ell(a,b)\) be the maximum number of strict steps in a chain from
\(a\) to \(b\); in particular, \(\ell(a,a)=0\). A \emph{bounded poset}
is a poset with a least element \(\bot\) and a greatest element
\(\top\), whose elements other than \(\bot\) and \(\top\) are called
\emph{interior elements}. Let \(P\) be a bounded poset of height
\(k\). For \(a\in P\), the \emph{rank} of \(a\) is
\(r(a)=\ell(\bot,a)\). Thus \(r(\bot)=0\), \(r(\top)=k\), and every
interior element has rank between \(1\) and \(k-1\). Clearly, any two
distinct elements of the same rank are incomparable. An element
\(p\in P\) is \emph{rigid} if it lies on a \(k\)-step chain from
\(\bot\) to \(\top\) in \(P\), equivalently, if
\(\ell(\bot,p)+\ell(p,\top)=k\).

Let \(P\) be a poset and \(Q\subseteq P\).
By a \emph{subposet} of \(P\) on \(Q\) we mean a poset
\(Q\) whose order \(\leq_Q\) is a subrelation of \(\leq_P\), that is, such
that \(a \leq_Q b\) implies \(a \leq_P b\) for all \(a,b\in Q\).
The \emph{induced subposet}, in contrast, is obtained by fully
inheriting the order, so that \(a \leq_Q b\) if and only if \(a \leq_P b\) for
all \(a,b\in Q\); equivalently, \(\leq_Q = \leq_P \cap (Q\times Q)\).
No comparability is lost in the induced case, so every
induced subposet is a subposet, but the converse need not hold.
For example, let \(P = \{1<2<3\}\) and \(Q = \{1<3, 2<3\}\).
Then \(Q\) is a subposet of \(P\) but not an induced subposet.

With these notions in hand, we record the following basic properties of
interval semirings.

\begin{lemma}\label{lem:icksa1ak}
Let $P$ be a poset. Then the following statements hold:
\begin{enumerate}[\((\rm i)\)]
\item \(I(Q)\) is a subsemiring of \(I(P)\) whenever \(Q\) is a subposet
of \(P\).

\item \(I(C_k)\) is isomorphic to \(S(a_1\cdots a_k)\) for every
\(k\geq 1\).

\item For any $k\geq 2$, \(I(P)\) is \(k\)-nilpotent if and only if
\(\htp(P)<k\) if and only if \(I(C_k)\notin \Var(I(P))\).

\item \(I(P)\) is \(2\)-nil.

\item An element \((a,b)\in I(P)\) is an annihilator if and only if
\(a\) is a minimal element of \(P\) and \(b\) is a maximal element of
\(P\).

\item \(I(P)\) is subdirectly irreducible if and only if \(P\) is
bounded.
\end{enumerate}
\end{lemma}

\begin{proof}
The first five items are immediate from the construction of \(I(P)\),
and the last follows from Lemma~\ref{sinilflat} together with (iv)
and (v).
\end{proof}

Although chains are among the simplest posets, they are general enough
to accommodate every finite poset: every finite poset \(P\) embeds
into the chain \(C_{|P|-1}\) as a subposet. Using this fact, we obtain
the following conclusion.
\begin{lemma}\label{lem:interval}
For every poset \(P\), the semiring \(I(P)\) is isomorphic to a subsemiring of \(I(C_{|P|-1})\).
\end{lemma}

\begin{proof}
Since \(P\) is finite, we may construct a map
\[
j:P\to C_{|P|-1}
\]
as follows.
Choose a minimal point of \(P\) and assign it the number \(0\).
Remove this point from \(P\), choose a minimal point of the remaining poset,
and assign it the number \(1\).
Continuing in this way, remove points one by one and assign them the numbers
\(0,1,\dots,|P|-1\) in the order in which they are chosen.
Since a point is always chosen minimally among the remaining points,
we have $j$ is injective and \(j(p)<j(q)\) whenever \(p<q\).

Define a map $\varphi\colon I(P)\to I(C_{|P|-1})$ by $\varphi(0)=0$ and
\[
\varphi((p,q))=(j(p),j(q)),\quad p<q.
\]
It is easy to verify that  \(\varphi\) is also injective and
preserves both the addition and multiplication. Thus $\varphi$ is an injective semiring homomorphism.
Consequently,
\[
\{(j(p),j(q))\mid p<q\}\cup\{0\}
\]
is a subalgebra of \(I(C_{|P|-1})\) which is isomorphic to \(I(P)\), as required.
\end{proof}

To compare interval semirings arising from different posets, it is
useful to relate the underlying posets by maps that preserve strict
order. This motivates the following notion. A map \(f\colon P\to Q\)
between two posets is \emph{strictly order preserving} if \(p<q\)
implies \(f(p)<f(q)\). A family
\(\{f_j\colon P\to Q_j\}_{j\in J}\) of strictly order preserving maps
is said to \emph{jointly separate} \(P\) if any two distinct elements
of \(P\) have different images under at least one member of the family.

\begin{lemma}\label{lem:separation-hsp}
Suppose that a family \(\{f_j\colon P\to Q_j\}_{j\in J}\) of strictly order preserving maps
jointly separates \(P\). Then \(I(P)\in\mathsf{V}(I(Q_j)\mid j\in J)\).
\end{lemma}

\begin{proof}
It suffices to show that \(I(P)\) lies in the
homomorphic image of a subalgebra of the direct product of
\(\{I(Q_j)\}_{j\in J}\). Concretely, let \(D=\prod_{j\in J}I(Q_j)\). For each pair
\(p<q\) in \(P\), define
\[
v_{pq}=\bigl((f_j(p),f_j(q))\bigr)_{j\in J}\in D.
\]
By strict order preservation of the maps \(f_j\), every coordinate of \(v_{pq}\) is
nonzero. Let \(A\) be the subsemiring of \(D\) generated by the set
\[
\{v_{pq}\mid p,q\in P,\ p<q\}.
\]

Let \(Z\) be the subset of \(A\) consisting of all elements having at
least one zero coordinate. Since \(0\) absorbs both addition and
multiplication, \(Z\) is an ideal of \(A\). Hence we may form the ideal
quotient \(A/Z\), in which \(Z\) absorbs both addition and
multiplication. Define a map
\[
\varphi\colon I(P)\to A/Z
\]
by sending \(0\) to \(Z\) and \((p,q)\) to \(v_{pq}/Z\). Since
\(v_{pq}/Z=\{v_{pq}\}\),  we shall henceforth identify \(v_{pq}/Z\)
with \(v_{pq}\).

We first show that \(\varphi\) is  injective. If
\((p,q)\neq(p',q')\), then \(p\neq p'\) or \(q\neq q'\). Since the
family \(\{f_j\}\) jointly separates \(P\), some coordinate \(j\) has
\(f_j(p)\neq f_j(p')\) or \(f_j(q)\neq f_j(q')\). Hence
\(v_{pq}(j)\neq v_{p'q'}(j)\), and so $v_{pq}\neq v_{p'q'}$. Thus \(\varphi\) is injective.

To show that \(\varphi\) is a semiring homomorphism, it suffices, by
the construction of \(\varphi\), to verify that \(\varphi\) preserves
addition and multiplication between any two distinct elements
\((p,q),(r,s)\in I(P)\).

The injectivity argument for \(\varphi\) above gives
\[
\varphi\bigl((p,q)+(r,s)\bigr)=\varphi(0)=Z
=v_{pq}+v_{rs}=\varphi(p,q)+\varphi(r,s).
\]

It remains to check multiplicativity. If \(q=r\), then
\[
\varphi(p,q)\varphi(r,s)=v_{pq}v_{rs}=v_{ps}
=\varphi(p,s)=\varphi\bigl((p,q)(r,s)\bigr).
\]
If \(q\neq r\), then, since the family \(\{f_j\}\) jointly separates
\(P\), there exists \(j\in J\) with \(f_j(q)\neq f_j(r)\). For this
coordinate, we deduce that
\[
v_{pq}(j)v_{rs}(j)=(f_j(p),f_j(q))(f_j(r),f_j(s))=0,
\]
and so \(v_{pq}v_{rs}\in Z\). Hence
\[
\varphi\bigl((p,q)(r,s)\bigr)=\varphi(0)=Z
=v_{pq}v_{rs}=\varphi(p,q)\varphi(r,s).
\]

Finally, we note that, since \(\varphi\) is a homomorphism and maps
onto the generators \(\{v_{pq}\mid p,q\in P,\ p<q\}\) of \(A/Z\), it
is surjective. Therefore \(\varphi\colon I(P)\to A/Z\) is an
isomorphism. Thus \(I(P)\in\mathsf{V}(I(Q_j)\mid j\in J)\).
\end{proof}

\begin{remark}
Lemma~\ref{lem:interval} is a special case of
Lemma~\ref{lem:separation-hsp}, since a single injective strictly order preserving map
\(f\colon P\to Q\) is already enough to separate \(P\) itself.
\end{remark}

\section{Identity bases and finite subdirectly irreducible members}\label{sec:bases}

Write \(\V_k = \mathsf{V}(I(C_k)) = \mathsf{V}(S(a_1\cdots a_k))\) and
\(\V_\infty = \bigvee_{k\geq 1}\V_k\). Then
\[
\V_\infty
= \mathsf{V}(I(C_k):k\geq 1)
= \mathsf{V}(S(a_1a_2\cdots a_k):k\geq 1).
\]
It is easy to check that \(\V_\infty\) can also be generated by a
single algebra \(I(C_\infty)\), where
\[
C_\infty = \{0<1<2<\cdots<k<\cdots\}.
\]
Indeed, every \(I(C_k)\) is a subsemiring of \(I(C_{\infty})\); conversely, every finitely generated subsemiring of \(I(C_{\infty})\)
is contained in \(I(C_k)\) for some \(k\) (for example, take \(k\) to be the maximum of the right endpoints of all the intervals involved).
This implies the first result.
\begin{lemma}
\(\V_{\infty} = \mathsf{V}(I(C_k)\mid k\geq 1) = \mathsf{V}(I(C_{\infty}))\).
\end{lemma}
Hence the subscript \(\infty\) is appropriate. On the other hand, \(\V_{\infty}\) is indeed not finitely generated, as shown in the proposition below.

\begin{proposition}
\(\V_{\infty}\) is not finitely generated, yet it is contained in a finitely generated variety.
\end{proposition}

\begin{proof}
Suppose, for contradiction, that \(\V_{\infty}\) is generated by a finite ai-semiring \(S\).
Since every interval semiring \(S(I(C_k))\) is \(2\)-nil, that is, satisfies the identity \(x^2\approx 0\),
\(S\) also satisfies this identity. Hence the multiplicative semigroup of \(S\) is a finite nil semigroup, and is therefore nilpotent (see \cite[Lemma 3.1.16]{sapir14}).
Assume that \(S\) is \(k\)-nilpotent for some $k$. Then \(S\) satisfies the \(k\)-nilpotent identity \(x_1\cdots x_k\approx 0\). In particular, \(I(C_k)\in \mathsf{V}(S)\),
so \(I(C_k)\) also satisfies this identity, a contradiction. Therefore, \(\V_{\infty}\) is not finitely generated.

On the other hand, by \cite[Proposition 5.3]{g}, every flat semiring \(S(a_1\cdots a_k)\) lies in \(\mathsf{V}(B_0)\), where \(B_0\) is a \(4\)-element ai-semiring. Thus \(\V_\infty \subseteq \mathsf{V}(B_0)\).
\end{proof}

It is readily verified that every flat semiring \(S(a_1\cdots a_k)\)
(or, equivalently, the interval semiring \(I(C_\infty)\)) satisfies the following identities:
\begin{gather}
x^2+y\approx x^2y\approx yx^2\approx x^2
  \qquad(x^2\approx 0\text{ for short})\label{eqx2free}\\
x+z_1x\approx x+xz_2\approx x+z_1xz_2\approx x^2\label{eqc1}\\
yx+xz\approx x^2\label{eqc2}\\
y_1x+z_1xz_2\approx xy_2+z_1xz_2\approx x^2\label{eqc3}\\
y_1xy_2+z_1xz_2\approx (y_1+z_1)x(y_2+z_2)\label{eqc4}
\end{gather}
Let \(\Delta_\infty\) be the set of identities~\eqref{eq:flat}--\eqref{eqc4}.
One of the main tasks that follows is to show that \(\Delta_{\infty}\)
is an identity basis for \(\mathcal{V}_{\infty}\).
To this end, we first gradually reveal relevant information about the
variety \(\Mod(\Delta_{\infty})\).

\begin{lemma}\label{localfinite}
Every  $n$-generated algebra in $\Mod(\Delta_{\infty})$ is $(n+1)$-nilpotent and finite.
In particular, the variety $\Mod(\Delta_{\infty})$ is locally finite.
\end{lemma}

\begin{proof}
First, taking $y=xy$ and $z=yx$ in~\eqref{eqc2}, together with~\eqref{eqx2free}, we deduce that
\begin{equation}\label{eqxyxx2free}
\Delta_{\infty} \vDash xyx\approx x^2\approx 0.
\end{equation}

Let $S=\langle a_1,\ldots,a_n\rangle$ be an $n$-generated algebra in $\Mod(\Delta_{\infty})$.
Then $S$ satisfies the identities~\eqref{eqxyxx2free}.
Let $H$ be the multiplicative semigroup of $S$ generated by $\{a_1,\ldots,a_n\}$.

Since $H$ satisfies the identity~\eqref{eqxyxx2free}, every nonzero element of $H$ is a product of distinct elements from $\{a_1,\ldots,a_n\}$.
Hence $H$ is $(n+1)$-nilpotent and so it is finite; more precisely,
\[
|H|\leq \sum_{k=1}^{n}\frac{n!}{(n-k)!}+1.
\]
Since $S$ is an ai-semiring, addition is commutative and idempotent.
Therefore every element of $S$ can be written as a sum of distinct elements of $H$, where the order of the summands is irrelevant. Consequently,
\[
|S|\leq 2^{|H|},
\]
which shows that $S$ is finite. Hence $\Mod(\Delta_{\infty})$ is locally finite.
\end{proof}

\begin{lemma}\label{lem:contexts}
Let $S$ be a finite subdirectly member in $\Mod(\Delta_{\infty})$.
Then $S$ is a nilpotent flat semiring a unique annihilator $\omega$ and for every $t\neq0$
there exist a unique pair $(\lambda_t,\rho_t)\in S^1\times S^1 $ such that
\begin{equation}\label{lttrtw}
\lambda_t t \rho_t = \omega.
\end{equation}
\end{lemma}

\begin{proof}
By Lemmas~\ref{siflat}, \ref{sinilflat} and \ref{localfinite}, it follows
that \(S\) is a nilpotent flat semiring with a unique annihilator
\(\omega\), and that for every \(t\neq 0\) there exist
\(\lambda_t,\rho_t\in S^1\) such that
\[
\lambda_t t \rho_t = \omega.
\]
Suppose that \(\lambda_t',\rho_t'\in S^1\) also satisfy
\(\lambda_t' t \rho_t' = \omega\).
By identities~\eqref{eqc1}--\eqref{eqc4}, it follows that
\((\lambda_t,\rho_t)=(\lambda_t',\rho_t')\).
Hence the pair \((\lambda_t,\rho_t)\) is uniquely determined.
\end{proof}

Next, we only need to characterize the finite subdirectly irreducible members of the variety~$\mathcal{V}_{\infty}$.

\begin{proposition}\label{prop:representation}
Every finite subdirectly irreducible member of $\Mod(\Delta_\infty)$ is isomorphic to $I(P)$ for a bounded poset $P$.
Conversely, every such $I(P)$ is a finite subdirectly irreducible and belongs to $\mathcal{V}_{\infty}$.
\end{proposition}

\begin{proof}
Let \(S\) be a finite
subdirectly irreducible member of \(\mathcal{V}_{\infty}\).
Define the prefix set $P_S$ of $S$ by
\[
 P_S = \{p\in S^1\mid (\exists\,q\in S^1)~pq=\omega\}.
\]
It is easy to verify that each $p\in P_S$ has a unique $R(p)\in S^1$ with $pR(p)=\omega$.
In particular, $R(1)=\omega$ and $R(\omega)=1$.

Define a partial order on $P_S$ by $p\leq r$ if $pt=r$ for some $t\in S^1$.
It is easy to verify that $P_S$ is a bounded poset with least element $1$ and greatest element $\omega$.

Define the map $\eta\colon S\to I(P_S)$ by
\[
 \eta(0)=0,\qquad \eta(t)=(\lambda_t,\lambda_t t) \quad(t\ne0).
\]

For $p<r$, choose $t\ne0$ with $pt=r$. Then
$ptR(r)=rR(r)=\omega$, so uniqueness in \eqref{lttrtw} gives $p=\lambda_t$ and $R(r)=\rho_t$.
Thus
\[
\eta(t)=(\lambda_t,\lambda_t t)=(p,pt)=(p,r).
\]
This proves surjectivity.

If $\eta(t)=\eta(t')$, then
\[
(\lambda_t,\lambda_t t)=(\lambda_{t'},\lambda_{t'}t')
\]
and so $\lambda_t t=\lambda_{t'}t'=\lambda_t t\neq0$.
By $0$-cancellativity, $t=t'$. Thus $\eta$ is injective.

Let $t,u\in S\setminus\{0\}$.
A direct computation gives
\[
\eta(tu)=\begin{cases}
  (\lambda_{tu},\lambda_{tu}tu),& tu\neq0,\\
  0,& tu=0,
\end{cases}
\qquad
\eta(t)\eta(u)=\begin{cases}
  (\lambda_t,\lambda_u u),& \lambda_t t=\lambda_u,\\
  0,& \lambda_t t\neq\lambda_u.
\end{cases}
\]
Consider the following two cases.

\textbf{Case 1.} $tu=0$. We claim that $\eta(t)\eta(u)=0$; otherwise, $\lambda_t t=\lambda_u$ and so
\[
\omega=\lambda_u u\rho_u=\lambda_t(tu)\rho_u=\lambda_t 0\rho_u=0,
\]
a contradiction. Thus $\eta(t)\eta(u)=0=\eta(0)=\eta(tu)$.

\textbf{Case 2.} $tu\neq0$. Then
\[
\omega=\lambda_{tu}(tu)\rho_{tu}=\lambda_t t\rho_t=\lambda_u u\rho_u.
\]
By~\eqref{lttrtw}, it follows that
\[
\lambda_{tu}=\lambda_t,\quad u\rho_{tu}=\rho_t,\quad \lambda_{tu}t=\lambda_u,\quad \rho_{tu}=\rho_u.
\]
Thus $\lambda_t t=\lambda_{tu}t=\lambda_u$. Then
\[
\eta(t)\eta(u)=(\lambda_t,\lambda_u u)=(\lambda_{tu},\lambda_{tu}tu)=\eta(tu).
\]
Thus we have shown that
\[
\eta(t)\eta(u)=\eta(tu)
\]
for all $t,u\in S\setminus\{0\}$. Since $\eta(0)=0$, a routine verification shows that $\eta$ preserves all products and the flat addition, so $\eta$ is a semiring homomorphism.
In summary, $\eta$ is a semiring isomorphism, and so $S\cong I(P_S)$. Hence take $P=P_S$.

Conversely, by Lemma~\ref{sinilflat}, every interval semiring \(I(P)\) of a bouned poset $P$ is a finite subdirectly irreducible
member of \(\mathcal{V}_{\infty}\). 
\end{proof}

\begin{remark}
In fact, the proof of Proposition~\ref{prop:representation} can be
further used to show that every subdirectly irreducible member of
\(\V_{\infty}\) is still exactly an interval semiring \(I(P)\) for
some bounded poset \(P\). However, in this case, the poset \(P\)
need not be finite.
\end{remark}

\begin{theorem}\label{thm:vinftyfb}
  The variety $\V_{\infty}$ is finitely based. More precisely,
  \[
  \V_{\infty}=\Mod(\Delta_{\infty}).
  \]
\end{theorem}

\begin{proof}
Since every flat semiring $S(a_1\cdots a_k)$ satisfies every identity in $\Delta_{\infty}$, and since $\V_{\infty}=\mathsf{V}(S(a_1\cdots a_k)\mid k\geq 1)$, it follows that $\V_{\infty}$ is a subvariety of $\Mod(\Delta_{\infty})$.
Conversely, Lemma~\ref{localfinite} tells us that $\Mod(\Delta_{\infty})$ is locally finite, so it can be generated by its finite subdirectly irreducible members. Proposition~\ref{prop:representation} shows that every such member is isomorphic to some interval semiring $I(P)$ and belongs to $\V_{\infty}$. Thus $\V_{\infty}=\Mod(\Delta_{\infty})$, and hence $\V_{\infty}$ is finitely based.
\end{proof}

To settle the finite basis problem for \(\V_k\), we refine
\(\Delta_\infty\) by adding two identities. For each \(k\geq 1\), let
\(\Delta_k\) be \(\Delta_\infty\) together with the following two
identities \(N_k\) and \(U_k\):
\[
N_k\colon x_1x_2\cdots x_{k+1}\approx z^2.
\]
\[
U_k\colon x_1\cdots x_k+y_1\cdots y_k
  \approx(x_1+y_1)\cdots(x_k+y_k).
\]

It is easy to see that \(S(a_1\cdots a_k)\) satisfies the identities
\(U_k\) and \(N_k\). Moreover, by the definition of \(I(P)\), the
following lemma is immediate; we omit the routine verification.

\begin{lemma}\label{lem:height-u}
Let \(P\) be a poset and \(k\geq 1\).
\begin{enumerate}[\((\rm i)\)]
\item \(I(P)\vDash N_k\) if and only if \(\htp(P)\leq k\).
\item \(I(P)\vDash U_k\) if and only if \(P\) has at most one chain
of \(k\) strict steps.
\end{enumerate}
\end{lemma}

Next, we use Lemma~\ref{lem:separation-hsp} to show that every
interval semiring satisfying \(N_k\) and \(U_k\) belongs to the
variety \(\V_k\).

\begin{lemma}\label{lem:ipck}
Let \(k\geq 1\) and let \(P\) be a bounded poset of height at most
\(k\) that contains at most one chain of \(k\) strict steps. Then
\(I(P)\in\V_k\).
\end{lemma}

\begin{proof}
Since \(\V_k=\mathsf{V}(I(C_k))\), Lemma~\ref{lem:separation-hsp} shows
that, in order to prove \(I(P)\in \V_k\), it suffices to find a family
of strictly order preserving maps \(P\to C_k\) that jointly separates \(P\).

We begin with the map \(f_0\colon P\to C_k\) defined by \(f_0(p)=r(p)\).
Clearly, \(f_0\) is strictly order preserving; consequently, it separates any two points of
different ranks. Thus it remains only to separate distinct interior
points of the same rank.

To this end, for each nonrigid interior point \(a\), define a map
\(f_a\colon P\to C_k\) by
\[
 f_a(p)=
 \begin{cases}
 \max\{r(p),\, r(a)+1+\ell(a,p)\}, & a\leq p,\\[2mm]
 r(p), & a\not\leq p.
 \end{cases}
\]
We first verify that \(f_a\) is well defined. Since \(r(p)\leq k\) and
\[
r(a)+1+\ell(a,p)\leq r(a)+1+\ell(a,\top)\leq k,
\]
the latter inequality holding because \(a\) is nonrigid, we have
\(f_a(p)\leq k\) for every \(p\in P\). Thus \(f_a\) takes values in
\(C_k\).

Next we show that \(f_a\) is  strictly order preserving. Let \(p,q\in P\) with \(p<q\).
If \(a\not\leq p\), then
\[
f_a(p)=r(p)<r(q)\leq f_a(q).
\]
If \(a\leq p\), then \(a\leq q\), and $r(q)\geq r(p)+1$, $\ell(a,q)\geq \ell(a,p)+1$.
Thus $f_a(p)<f_a(q)$.
This proves that \(f_a\) is strictly order preserving.

Now let \(p\) and \(q\) be distinct points of \(P\) having the same rank.
Because \(P\) contains at most one \(k\)-step chain, at least one of
them, say \(p\), is nonrigid. Points of the same rank are incomparable,
so \(p\not\leq q\). Therefore \(f_p(q)=r(q)\), whereas
\[
f_p(p)=\max\{r(p),\,r(p)+1+\ell(p,p)\}=r(p)+1.
\]
As \(r(p)=r(q)\), it follows that \(f_p(p)\neq f_p(q)\).
Consequently, every pair of distinct points of \(P\) is separated by
some strictly order preserving map $P\to C_k$. Hence \(I(P)\in \V_k\).
\end{proof}

\begin{proposition}\label{prop:bounded-separation}
Let \(k\geq 1\). Then, up to isomorphism, every finite subdirectly
irreducible member of \(\mathcal V_k\) is exactly an interval semiring
\(I(P)\), where \(P\) is a bounded poset of height at most \(k\)
containing at most one chain of length \(k\).
\end{proposition}

\begin{proof}
By Proposition~\ref{prop:representation}, up to isomorphism, every
finite subdirectly irreducible member of \(\mathcal V_k\) is an interval
semiring \(I(P)\) for a bounded poset \(P\).

Let \(I(P)\) be an interval semiring in \(\mathcal V_k\). Since
\(\mathcal V_k\) satisfies the identities \(N_k\) and \(U_k\), it
follows from Lemma~\ref{lem:height-u} that \(P\) is a bounded poset of
height at most \(k\) containing at most one chain of length \(k\).

Conversely, let \(P\) be such a poset. By Lemma~\ref{lem:ipck}, it
follows that \(I(P)\in\V_k\). This completes the proof.
\end{proof}

\begin{theorem}\label{thm:bases}
For each $k\geq 1$, the variety $\V_k$ is finitely based. More precisely,
$\V_k = \Mod(\Delta_k)$.
\end{theorem}
\begin{proof}
Since $S(a_1\cdots a_k)$ satisfies every identity in $\Delta_k$,
it follows that the variety $\V_k=\mathsf{V}(S(a_1\cdots a_k))$ is a subvariety of $\Mod(\Delta_k)$.
Conversely, since $\Mod(\Delta_k)$ is a subvariety of $\Mod(\Delta_{\infty})$, it follows from
Lemma~\ref{localfinite} that $\Mod(\Delta_k)$ is locally finite, and so it can be generated by its finite subdirectly irreducible members.
By Proposition~\ref{prop:representation} and Lemma~\ref{lem:height-u},
every such member is isomorphic to an interval semiring $I(P)$ for some bounded poset $P$ with $\htp(P)\leq k$,
and with at most one $k$-step chain.
Finally, by Proposition~\ref{prop:bounded-separation}, such $I(P)$ lies in $\V_k$.
This shows that $\V_k = \Mod(\Delta_k)$, and hence $\V_k$ is finitely based.
\end{proof}

\begin{proposition}
Let $W$ be a nonempty set of linear words. Then $S(W)$ is finitely based.
More precisely, $\mathsf{V}(S(W)) = \V_{k}$, where $k = \sup\{\ell(\bw)\mid \bw\in W\}\leq \infty$.
\end{proposition}

\begin{proof}
By \cite[Proposition 3.12]{grz}, we have
\[
\mathsf{V}(S(W)) = \mathsf{V}(S(\mathbf w)\mid \mathbf w\in W).
\]

On the other hand, for linear words \(\mathbf u\) and \(\mathbf v\),
it is easy to verify that \(S(\mathbf u)\in \mathsf{V}(S(\mathbf v))\)
if and only if \(\ell(\mathbf u)\leq \ell(\mathbf v)\).
We now distinguish whether \(k\) is finite.

If \(k<\infty\), then it is easy to verify that
\[
\mathsf{V}(S(W)) = \mathsf{V}(S(a_1\cdots a_k)) = \V_k.
\]

If \(k=\infty\), then \(S(a_1\cdots a_n)\in \mathsf{V}(S(W))\) for all
\(n\geq 1\). Therefore
\[
\mathsf{V}(S(W)) = \V_{\infty}.
\]

Finally, by Theorems~\ref{thm:vinftyfb} and~\ref{thm:bases}, the
conclusion follows.
\end{proof}

\begin{remark}
Compared with the commutative case, $S_c(a_1\cdots a_k)$ is finitely based if and only if $k\leq 2$.
Moreover, the variety generated by all $S_c(a_1\cdots a_k)$ is also nonfinitely based; see~\cite{gjrz}.
In particular, for every positive integer $k\geq 3$, there exists an exactly $k$-nilpotent flat semiring that is finitely based, and there also exists an exactly $k$-nilpotent flat semiring that is nonfinitely based.
\end{remark}

\section{Bounded-height varieties}\label{sec:bounded}

For $k\geq1$, let  $\B_k$ be the subvariety of $\V_{\infty}$ defined by the identity $N_{k}$; that is
\[
 \B_k=\V_{\infty}\cap\Mod(x_1\cdots x_{k+1}\approx x^2).
\]

\begin{lemma}\label{lem:sibk}
The variety \(\mathcal B_k\) is generated by all interval semirings of
bounded posets of height at most \(k\). Moreover, it is generated by
all interval semirings of bounded posets of height exactly \(k\).
\end{lemma}

\begin{proof}
By Lemma~\ref{localfinite}, the variety \(\mathcal B_k\) is locally finite,
so it is generated by its finite subdirectly irreducible members.
Furthermore, by Proposition~\ref{prop:representation} and
Lemma~\ref{lem:height-u}(i), every such member is precisely an interval
semiring of a bounded poset of height at most \(k\).

On the other hand, every bounded poset of height less than \(k\) can be
embedded naturally into a bounded poset of height \(k\), for instance by
adding suitable points at the bottom. Therefore, the desired conclusion
follows.
\end{proof}

In fact, \(\B_k\) can be generated by a single interval semiring of a
bounded poset of height \(k\), constructed as follows. For each \(k\geq 2\),
define \(Q_k\) as the ordinal sum
\[
 \{\bot\}\ \oplus\ \{a_1,b_1\}\ \oplus\ \{a_2,b_2\}\oplus\ \cdots\ \oplus
 \{a_{k-1},b_{k-1}\}\ \oplus\ \{\top\},
\]
where each two-point set is an antichain, so that the strict order of
\(Q_k\) is exactly
\[
 \bot<a_1,b_1<a_2,b_2<\cdots<a_{k-1},b_{k-1}<\top.
\]
For example, the Hasse diagram of \(Q_4\) is given in Figure~\ref{fig:Q4}.
\begin{figure}[htbp]
\centering
\begin{tikzpicture}[
  scale=0.8,
  elem/.style={circle, fill=black, inner sep=1.6pt},
  lab/.style={font=\small}
]

\node[elem] (bot) at (0,0) {};
\node[elem] (a1)  at (-1,0.9) {};
\node[elem] (b1)  at ( 1,0.9) {};
\node[elem] (a2)  at (-1,1.8) {};
\node[elem] (b2)  at ( 1,1.8) {};
\node[elem] (a3)  at (-1,2.7) {};
\node[elem] (b3)  at ( 1,2.7) {};
\node[elem] (top) at (0,3.6) {};

\node[lab] at (0,-0.35)   {\(\bot\)};
\node[lab] at (-1.35,0.9) {\(a_1\)};
\node[lab] at ( 1.35,0.9) {\(b_1\)};
\node[lab] at (-1.35,1.8) {\(a_2\)};
\node[lab] at ( 1.35,1.8) {\(b_2\)};
\node[lab] at (-1.35,2.7) {\(a_3\)};
\node[lab] at ( 1.35,2.7) {\(b_3\)};
\node[lab] at (0,3.95)    {\(\top\)};

\draw[thick] (bot) -- (a1);
\draw[thick] (bot) -- (b1);

\draw[thick] (a1) -- (a2);
\draw[thick] (a1) -- (b2);
\draw[thick] (b1) -- (a2);
\draw[thick] (b1) -- (b2);

\draw[thick] (a2) -- (a3);
\draw[thick] (a2) -- (b3);
\draw[thick] (b2) -- (a3);
\draw[thick] (b2) -- (b3);

\draw[thick] (a3) -- (top);
\draw[thick] (b3) -- (top);

\end{tikzpicture}
\caption{The poset \(Q_4\).}
\label{fig:Q4}
\end{figure}
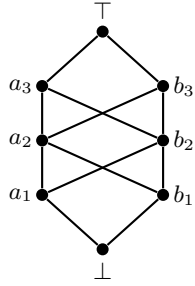

\begin{proposition}\label{prop:finite}
For every \(k\geq1\), the variety \(\B_k\) is finitely generated.
More precisely, \(\B_k=\mathsf{V}(I(Q_k))\), where
\(|I(Q_k)|=2k^2-2k+2\).
\end{proposition}

\begin{proof}
Since \(Q_k\) is a bounded poset of height \(k\), it follows from
Lemma~\ref{lem:sibk} that \(I(Q_k)\in \B_k\).

Conversely, by Lemmas~\ref{lem:separation-hsp} and~\ref{lem:sibk}, it
suffices to verify that every finite bounded poset \(P\) of height \(k\)
admits a family of strictly order preserving maps \(P\to Q_k\) that jointly separate \(P\).

Define a strictly order preserving map \(f\colon P\to Q_k\) by fixing \(\bot\) and \(\top\),
and setting
\[
f(p)=a_{\ell(\bot,p)}
\]
for every interior point \(p\). The map \(f\) separates the bounds from
the interior points, and it separates interior points of different ranks.

Now let \(a\) be an interior point that has another interior point of the
same rank. Alter only its image under \(f\) to
\(b_{\ell(\bot,a)}\), obtaining another strictly order preserving map \(f_a\). Indeed,
comparisons always involve points of different levels, so strict order preservation is
preserved. The baseline map \(f\) separates different ranks and the
bounds. For distinct points \(a,b\) of equal rank, we have
\[
f_a(a)=b_{\ell(\bot,a)}
\neq a_{\ell(\bot,a)}
=f_a(b),
\]
so \(f_a\) separates \(a\) and \(b\). This completes the proof.
\end{proof}

\begin{proposition}\label{prop:proper}
Every proper subvariety $\mathcal W$ of $\V_{\infty}$ is contained in some
$\B_k$ with $k\geq1$, and every $\B_k$ is proper. Consequently,
\[
 \Lat(\V_{\infty})\setminus\{\V_{\infty}\}
   =\bigcup_{h\geq1}\Lat(\B_k).
\]
\end{proposition}

\begin{proof}
For each \(k\geq 1\), since \(I(C_{k+1})\in \V_{\infty}\setminus \B_{k}$,
it follows that every $\B_k$ is a proper subvariety of $\V_{\infty}$.

Next, suppose that $\W$ is any proper subvariety of $\V_{\infty}$.
If for every positive integer $k$ there exists an interval semiring
$I(P)\in \W$ with $\htp(P)\geq k$, then $I(C_k)$ is a subalgebra
of $I(P)$, which shows that $I(C_k)\in \W$ for every $k\geq 1$.
Hence $\W = \V_{\infty}$, a contradiction.

Therefore there exists a positive integer $k$ such that every
interval semiring $I(P)$ in $\W$ has $\htp(P)\leq k$.
Combining this with Lemma~\ref{lem:icksa1ak}(iii), we obtain that
$I(P)$ is $(k+1)$-nilpotent. Thus $\W$ satisfies the identity
$N_k$, and so $\W$ is a subvariety of $\B_k$. In particular,
\[
 \Lat(\V_{\infty})\setminus\{\V_{\infty}\}
   =\bigcup_{k\geq1}\Lat(\B_k).
\]
This completes the proof.
\end{proof}

\begin{proposition}\label{prop:sandwich}
For every \(k\geq 2\), the variety \(\mathcal V_k\) is a proper
subvariety of \(\mathcal B_k\), and \(\mathcal B_k\) is a maximal
subvariety of \(\mathcal V_{k+1}\). In particular, \(\V_{\infty}\)
admits the following infinite strictly increasing chain of subvarieties:
\[
\V_2<\B_2<\V_3<\B_3<\cdots<\V_{\infty}.
\]
\end{proposition}

\begin{proof}
Since \(\htp(I(C_k))=k\), Lemma~\ref{lem:sibk} implies that
\(I(C_k)\in\mathcal B_k\). Hence
\(\mathsf V(I(C_k))\) is a subvariety of \(\mathcal B_k\).

Now let \(P\) be a poset of height \(k\) containing two chains of length
\(k\). Then by Lemma~\ref{lem:sibk} again, \(I(P)\in\mathcal B_k\).
By Lemma~\ref{lem:height-u}(ii),
\(I(P)\) does not satisfy the identity \(U_k\). Therefore
\(\mathcal B_k\) does not satisfy \(U_k\), whereas it is easy to verify
that \(I(C_k)\) satisfies \(U_k\). This shows that \(\mathcal V_k\) is a
proper subvariety of \(\mathcal B_k\).

Conversely, by Lemma~\ref{lem:sibk}, \(\mathcal B_k\) is generated by
all interval semirings of bounded posets of height at most \(k\), and by
Proposition~\ref{prop:bounded-separation} every such semiring lies in
\(\mathcal V_{k+1}\). Hence \(\mathcal B_k\) is a subvariety of
\(\mathcal V_{k+1}\). Using \(I(C_{k+1})\), one sees that this subvariety is
proper.

Finally, by Proposition~\ref{prop:bounded-separation} again, a
subvariety \(\mathcal W\) of \(\mathcal V_{k+1}\) is proper if and only
if \(\mathcal W\) contains no interval semiring induced by a bounded
poset \(P\) of height at least \(k+1\). Hence \(\mathcal B_k\) is a
maximal subvariety of \(\mathcal V_{k+1}\).
\end{proof}

\section{The Cross variety $\B_3$}\label{sec:crossb3}
Consider the following seven bounded posets of height at most three, whose Hasse diagrams are collected in Figure~\ref{fig:seven-posets}.
We shall prove that every bounded poset \(P\) of height at most three can be handled by means of these seven bounded posets and suitable strictly order preserving maps that jointly separate \(P\).

\begin{figure}[htbp]
\centering
\begin{tikzpicture}[
  scale=0.8,
  elem/.style={circle, fill=black, inner sep=1.5pt},
  lab/.style={font=\small},
  edge/.style={thick}
]

\begin{scope}[xshift=0cm]
  \node[elem] (c1b) at (0,0) {};
  \node[elem] (c1t) at (0,2.4) {};
  \node[lab] at (0,-0.35) {\(\bot\)};
  \node[lab] at (0,2.75)  {\(\top\)};
  \draw[edge] (c1b) -- (c1t);
  \node[lab] at (0,-1.0) {\(C_1\)};
\end{scope}

\begin{scope}[xshift=2.8cm]
  \node[elem] (c2b) at (0,0) {};
  \node[elem] (c2a) at (0,1.2) {};
  \node[elem] (c2t) at (0,2.4) {};
  \node[lab] at (0,-0.35) {\(\bot\)};
  \node[lab] at (0.3,1.2) {\(a\)};
  \node[lab] at (0,2.75)  {\(\top\)};
  \draw[edge] (c2b) -- (c2a);
  \draw[edge] (c2a) -- (c2t);
  \node[lab] at (0,-1.0) {\(C_2\)};
\end{scope}

\begin{scope}[xshift=5.6cm]
  \node[elem] (v2b)  at (0,0) {};
  \node[elem] (v2a)  at (-0.8,1.2) {};
  \node[elem] (v2ap) at (0.8,1.2) {};
  \node[elem] (v2t)  at (0,2.4) {};
  \node[lab] at (0,-0.35) {\(\bot\)};
  \node[lab] at (-1.25,1.2) {\(a\)};
  \node[lab] at (1.25,1.2)  {\(a'\)};
  \node[lab] at (0,2.75)    {\(\top\)};
  \draw[edge] (v2b) -- (v2a);
  \draw[edge] (v2b) -- (v2ap);
  \draw[edge] (v2a) -- (v2t);
  \draw[edge] (v2ap) -- (v2t);
  \node[lab] at (0,-1.0) {\(E_2\)};
\end{scope}

\begin{scope}[xshift=8.8cm]
  \node[elem] (c3b) at (0,0) {};
  \node[elem] (c3l) at (0,0.8) {};
  \node[elem] (c3u) at (0,1.6) {};
  \node[elem] (c3t) at (0,2.4) {};
  \node[lab] at (0,-0.35) {\(\bot\)};
  \node[lab] at (0.3,0.8) {\(\ell\)};
  \node[lab] at (0.3,1.6) {\(u\)};
  \node[lab] at (0,2.75)  {\(\top\)};
  \draw[edge] (c3b) -- (c3l);
  \draw[edge] (c3l) -- (c3u);
  \draw[edge] (c3u) -- (c3t);
  \node[lab] at (0,-1.0) {\(C_3\)};
\end{scope}

\begin{scope}[xshift=0cm, yshift=-7.0cm]
  \node[elem] (l3b)  at (0,0) {};
  \node[elem] (l31)  at (-1.0,1.3) {};
  \node[elem] (l32)  at (1.0,1.3)  {};
  \node[elem] (l3u)  at (0,2.6)    {};
  \node[elem] (l3t)  at (0,3.9)    {};
  \node[lab] at (0,-0.35) {\(\bot\)};
  \node[lab] at (-1.45,1.3) {\(\ell_1\)};
  \node[lab] at (1.45,1.3)  {\(\ell_2\)};
  \node[lab] at (0.3,2.6)   {\(u\)};
  \node[lab] at (0,4.25)    {\(\top\)};
  \draw[edge] (l3b) -- (l31);
  \draw[edge] (l3b) -- (l32);
  \draw[edge] (l31) -- (l3u);
  \draw[edge] (l32) -- (l3u);
  \draw[edge] (l3u) -- (l3t);
  \node[lab] at (0,-1.0) {\(L_3\)};
\end{scope}

\begin{scope}[xshift=4.6cm, yshift=-7.0cm]
  \node[elem] (u3b)  at (0,0) {};
  \node[elem] (u3l)  at (0,1.3) {};
  \node[elem] (u31)  at (-1.0,2.6) {};
  \node[elem] (u32)  at (1.0,2.6)  {};
  \node[elem] (u3t)  at (0,3.9)    {};
  \node[lab] at (0,-0.35) {\(\bot\)};
  \node[lab] at (0.3,1.3) {\(\ell\)};
  \node[lab] at (-1.45,2.6) {\(u_1\)};
  \node[lab] at (1.45,2.6)  {\(u_2\)};
  \node[lab] at (0,4.25)    {\(\top\)};
  \draw[edge] (u3b) -- (u3l);
  \draw[edge] (u3l) -- (u31);
  \draw[edge] (u3l) -- (u32);
  \draw[edge] (u31) -- (u3t);
  \draw[edge] (u32) -- (u3t);
  \node[lab] at (0,-1.0) {\(U_3\)};
\end{scope}

\begin{scope}[xshift=9.2cm, yshift=-7.0cm]
  \node[elem] (e3b)  at (0,0) {};
  \node[elem] (e31)  at (-1.2,1.3) {};
  \node[elem] (e32)  at (1.2,1.3)  {};
  \node[elem] (e33)  at (-1.2,2.6) {};
  \node[elem] (e34)  at (1.2,2.6)  {};
  \node[elem] (e3t)  at (0,3.9)    {};
  \node[lab] at (0,-0.35) {\(\bot\)};
  \node[lab] at (-1.65,1.3) {\(\ell_1\)};
  \node[lab] at (1.65,1.3)  {\(\ell_2\)};
  \node[lab] at (-1.65,2.6) {\(u_1\)};
  \node[lab] at (1.65,2.6)  {\(u_2\)};
  \node[lab] at (0,4.25)    {\(\top\)};
  \draw[edge] (e3b) -- (e31);
  \draw[edge] (e3b) -- (e32);
  \draw[edge] (e31) -- (e33);
  \draw[edge] (e32) -- (e34);
  \draw[edge] (e33) -- (e3t);
  \draw[edge] (e34) -- (e3t);
  \node[lab] at (0,-1.0) {\(E_3\)};
\end{scope}

\end{tikzpicture}
\caption{Hasse diagrams of the seven bounded posets.}
\label{fig:seven-posets}
\end{figure}
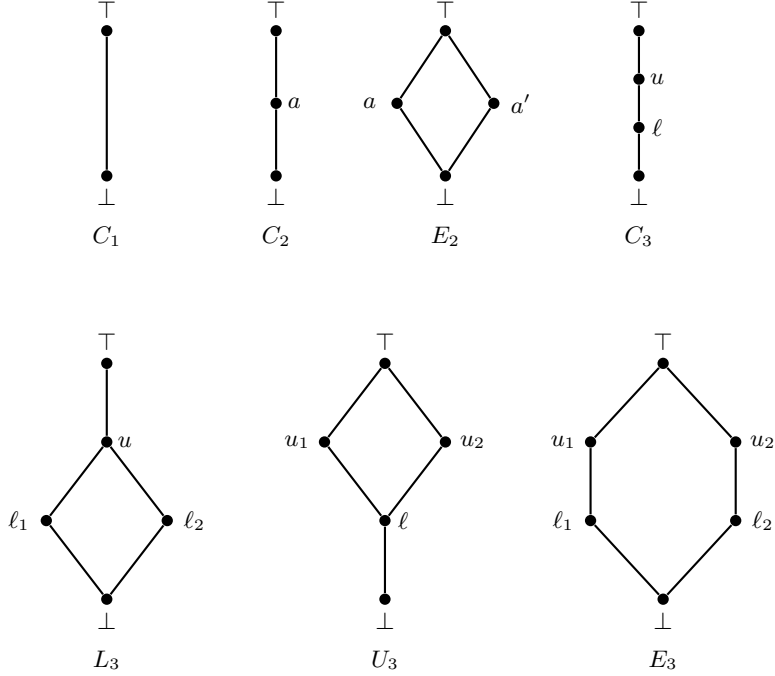

\begin{proposition}\label{prop:seven}
Let \(P\) be a bounded poset of height at most three,
and let \(\{Q_j\}_{j\in J}\) be the set of posets in Figure~\ref{fig:seven-posets}
that occur as subposets of \(P\).
Then
\[
\mathsf{V}(I(P)) = \mathsf{V}(I(Q_j)\mid j\in J).
\]
\end{proposition}

\begin{proof}
By Lemma~\ref{lem:icksa1ak}(i), $I(Q_j)$ is a subsemiring of $I(P)$ whenever $Q_j$ is a subposet of $P$.
This implies one containment
\[
\mathsf{V}(I(Q_j)\mid j\in J) \subseteq \mathsf{V}(I(P)).
\]

For the reverse containment, we construct strictly order preserving maps into the posets
of Figure~\ref{fig:seven-posets} that occur in \(P\), separating any
prescribed distinct pair of points of \(P\).
In what follows, all the maps we construct will fix the \(\bot\) and \(\top\), so it remains only to determine the images of the interior points.

If \(P\) has no interior points, it is \(C_1\). If its interior has
no comparisons, it is an antichain. One interior point \(a\) gives
\(C_2\); at least two points \(a\) and \(a'\) give an induced \(E_2\).
Arbitrary assignments of all interior points to the two interior
points of this \(E_2\) are strictly order preserving and, taken together, separate every pair.

Now suppose there is an interior comparison. Partition the
interior into \(L,U,M\), where \(L\) consists of points having an
interior successor, \(U\) of points having an interior predecessor,
and \(M\) of points having neither. No point belongs to both
\(L\) and \(U\), since such a point would lie in a four-step
chain between the bounds. Every interior comparison therefore
has the form \(\ell<u\) with \(\ell\in L,u\in U\), every point of \(L\)
and of \(U\) has a neighbor on the other side, and \(M\) has no
interior comparisons.

Choose one comparison \(\ell<u\). It gives an induced \(C_3\).
Send all of \(L\) to \(\ell\), all of \(U\) to \(u\), and each
point of \(M\) independently to either point. These maps are
strictly order preserving. They separate each bound from the interior, \(L\) from
\(U\), any point of \(M\) from a point of \(L\) or \(U\), and any
two distinct points of \(M\).

It remains to separate pairs within \(L\) or within \(U\). Fix
distinct \(\ell_1,\ell_2\in L\). We consider the following two cases.

\textbf{Case 1.} The points \(\ell_1\) and \(\ell_2\) have a common upper
bound \(u\). Then these three interior points induce \(L_3\).
Map every point of \(U\) to \(u\), keep \(\ell_1\) and \(\ell_2\) fixed,
map every other point of \(L\) to \(\ell_1\), and map every point of
\(M\) to \(u\). This gives a strictly order preserving map from \(P\) to \(L_3\)
separating \(\ell_1\) and \(\ell_2\).

\textbf{Case 2.} The points \(\ell_1\) and \(\ell_2\) have no common upper
bound. Choose upper neighbors \(u_1\) of \(\ell_1\) and \(u_2\) of \(\ell_2\).
These four points induce \(E_3\). Keep these four points fixed,
assign every other point of \(U\) to \(u_1\), every other point of
\(L\) to \(\ell_1\), and every point of \(M\) to \(u_1\). Then this map
is strictly order preserving and separates \(\ell_1\) and \(\ell_2\).

Dually, we can separate any distinct \(u_1,u_2\in U\). If they have
a common lower bound, we use \(U_3\); if they have no common lower
bound, we use \(E_3\).

This exhausts all pairs. By Lemma~\ref{lem:separation-hsp}, \(I(P)\)
lies in the variety generated by the interval semirings of the
induced bounded posets considered above. Hence the reverse
containment holds as well. Combining both containments, we obtain
the desired equality.
\end{proof}

In what follows, we investigate the finite basis problem for the subvarieties of \(\B_3\) and the structure of the lattice of subvarieties \(\mathcal{L}(\B_3)\).
We begin with the following two lemmas. The first one is well known, and we state it below without proof.
\begin{lemma}\label{prosubvariety}
Let \(\mathcal{W}\) be a maximal subvariety of \(\mathcal{V}_{\infty}\), and let \(\Sigma\) be a basis for \(\mathcal{V}_{\infty}\).
If an identity \(\sigma\) is satisfied by \(\mathcal{W}\) but not by \(\mathcal{V}_{\infty}\), then \(\Sigma\cup\{\sigma\}\) is a basis for \(\mathcal{W}\).
In particular, if \(\mathcal{V}_{\infty}\) is finitely based and has only finitely many subvarieties, then every subvariety of \(\mathcal{V}_{\infty}\) is finitely based.
\end{lemma}
The following lemma shows that the algebras \(I(L_3)\), \(I(U_3)\), and \(I(E_3)\) are distinguished by the corresponding identities.
\begin{lemma}\label{lem:identities}
For the four identities
\begin{align}
 xyz+uvz&\approx(x+u)(y+v)z,\tag{\(\alpha\)}\label{eq:alpha}\\
 xyz+xuv&\approx x(y+u)(z+v),\tag{\(\beta\)}\label{eq:beta}\\
 xyz+uvw&\approx xyz+uvw+xyw,\tag{\(\gamma\)}\label{eq:gamma}\\
 xyz+uvw&\approx xyz+uvw+xvw,\tag{\(\delta\)}\label{eq:delta}
\end{align}
the validity table for \((\alpha),(\beta),(\gamma),(\delta)\) on \(I(L_3)\), \(I(U_3)\), and \(I(E_3)\) is given in Table~\ref{tab:identities}, where \(1\) means ``satisfies'' and \(0\) means ``does not satisfy''.
\end{lemma}

\begin{table}[htbp]
\centering
\caption{The validity table.}
\label{tab:identities}
\begin{tabular}{l|cccc}
\hline
Algebra & \(\alpha\) & \(\beta\) & \(\gamma\) & \(\delta\)\\
\hline
\(I(L_3)\) & \(0\) & \(1\) & \(1\) & \(0\)\\
\(I(U_3)\) & \(1\) & \(0\) & \(0\) & \(1\)\\
\(I(E_3)\) & \(1\) & \(1\) & \(0\) & \(0\)\\
\hline
\end{tabular}
\end{table}

\begin{proof}
Let \(S\in \{I(L_3),I(U_3),I(E_3)\}\), and let \(\mathbf{u} \approx \mathbf{v}\) be one of the identities \((\alpha)\)--\((\delta)\).
By construction, every element of \(S^3\) is either \(0\) or \(\omega\), where \(\omega = (\bot,\top)\).
Moreover, the terms \(\mathbf{u}\) and \(\mathbf{v}\) are sums of words of length \(3\); hence, for every assignment,
\[
\varphi(\mathbf{u}),\varphi(\mathbf{v})\in \{0,\omega\}.
\]
Therefore, \(S\) satisfies the identity \(\mathbf{u}\approx \mathbf{v}\) if and only if either
\[
\varphi(\mathbf{u}) = \varphi(\mathbf{v}) = \omega
\quad\text{or}\quad
\varphi(\mathbf{u}) = \varphi(\mathbf{v}) = 0.
\]
Finally, put
\[
\Omega(S) = \{(a,b,c)\in S^3\mid abc=\omega\}.
\]
Then
\[
\Omega(I(L_3)) = \{((\bot,l_1),(l_1,u),(u,\top)),\ ((\bot,l_2),(l_2,u),(u,\top))\},
\]
\[
\Omega(I(U_3)) = \{((\bot,l),(l,u_1),(u_1,\top)),\ ((\bot,l),(l,u_2),(u_2,\top))\},
\]
\[
\Omega(I(E_3)) = \{((\bot,l_1),(l_1,u_1),(u_1,\top)),\ ((\bot,l_1),(l_1,u_2),(u_2,\top))\}.
\]
Thus, Table~\ref{tab:identities} follows by direct verification.
\end{proof}

\begin{proposition}
The variety \(\B_3\) has exactly \(11\) subvarieties, and its lattice of subvarieties is shown in Figure~\ref{fig:lat-b3}.
Every nontrivial subvariety of \(\B_3\) is generated by a nonempty subset of
\[
\{I(C_1),I(C_2),I(E_2),I(C_3),I(L_3),I(U_3),I(E_3)\}.
\]
Moreover, every subvariety of \(\B_3\) is finitely based; consequently, \(\B_3\) is a Cross variety.
\end{proposition}

\begin{proof}
By Lemma~\ref{localfinite}, the variety \(\B_3\) is locally finite.
Thus every subvariety of \(\B_3\) is generated by some finite subdirectly irreducible members.
By Proposition~\ref{prop:representation}, every such member is isomorphic to an interval semiring of a bounded poset of height at most \(3\).
By Proposition~\ref{prop:seven}, every nontrivial subvariety of \(\B_3\) is generated by a nonempty subset of
\[
\{I(C_1),I(C_2),I(E_2),I(C_3),I(L_3),I(U_3),I(E_3)\}.
\]

We first claim that
\begin{equation}\label{S(abc)}
  \mathsf{V}(I(C_1))<\mathsf{V}(I(C_2))<\mathsf{V}(I(E_2))<\mathsf{V}(I(C_3)).
\end{equation}

\begin{equation}\label{c3l3u3e3}
\mathsf{V}(I(C_3))<\mathsf{V}(I(L_3)),\ \mathsf{V}(I(U_3)),\ \mathsf{V}(I(E_3)).
\end{equation}

\begin{equation}\label{e3l3u3}
\mathsf{V}(I(E_3))\leq \mathsf{V}(I(L_3),I(U_3)) = \B_3.
\end{equation}

Indeed, \eqref{S(abc)} lists exactly the nontrivial subvarieties of \(V(S(abc))\), namely
\[
\mathsf{V}(S(a))<\mathsf{V}(S(ab))<\mathsf{V}(S(ab)\circ S(ab))<\mathsf{V}(S(abc)),
\]
as shown in~\cite{gr22}, where $S(ab)\circ S(ab)$ is isomorphic to $I(C_2)$.

\eqref{c3l3u3e3} is also clear, since \(L_3\), \(U_3\), and \(E_3\) have height \(3\) and all contain the subposet \(C_3\).
On the other hand, \(I(C_3)\) satisfies the identity \(U_3\):
\[
x_1x_2x_3+y_1y_2y_3\approx (x_1+y_1)(x_2+y_2)(x_3+y_3),
\]
whereas \(I(L_3)\), \(I(U_3)\), and \(I(E_3)\) do not satisfy this identity.

As for~\eqref{e3l3u3}, define maps \(\varphi_1\colon E_3\to L_3\) and \(\varphi_2\colon E_3\to U_3\) by fixing \(\bot\) and \(\top\) and setting
\[
\varphi_1(l_1)=l_1,\quad \varphi_1(l_2)=l_2,\quad \varphi_1(u_1)=\varphi_1(u_2)=u,
\]
\[
\varphi_2(l_1)=l_1,\quad \varphi_2(l_2)=l,\quad \varphi_2(u_1)=u_1,\quad \varphi_2(u_2)=u_2.
\]
Then \(\{\varphi_1,\varphi_2\}\) jointly separate \(E_3\).
By Lemma~\ref{lem:separation-hsp}, it follows that
\[
I(E_3)\in \mathsf{V}(I(L_3),I(U_3)).
\]
Also, \(\B_3 = \mathsf{V}(I(L_3),I(U_3),I(E_3)) = \mathsf{V}(I(L_3),I(U_3))\).

It remains to consider the proper subvarieties of \(\B_3\) that are larger than \(\mathsf{V}(I(C_3))\).
Such a subvariety is generated by a nonempty subset of
\[
\{I(L_3),I(U_3),I(E_3)\}.
\]
These are
\[
\mathsf{V}(I(L_3)),\ \mathsf{V}(I(U_3)),\ \mathsf{V}(I(E_3)),\
\mathsf{V}(I(L_3),I(E_3)),\ \mathsf{V}(I(U_3),I(E_3)).
\]
By Lemma~\ref{lem:identities}, \(\B_3\) has exactly \(11\) subvarieties; its Hasse diagram is shown in Figure~\ref{fig:lat-b3}.
Since \(\B_3\) is finitely based and has only finitely many subvarieties, Lemma~\ref{prosubvariety} implies that every subvariety of \(\B_3\) is finitely based.
\end{proof}

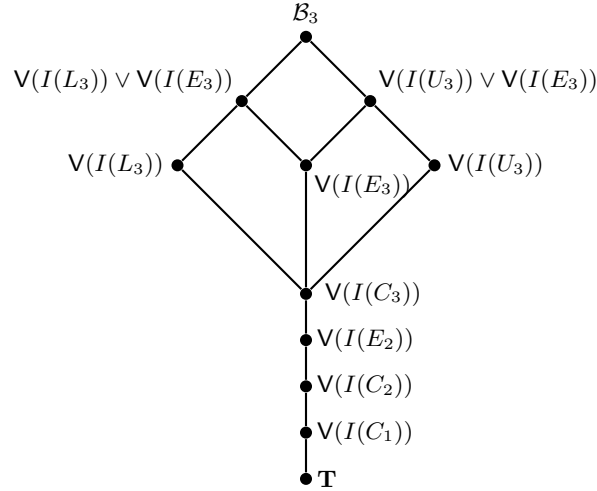
\begin{figure}[htbp]
\centering

\begin{tikzpicture}[
  scale=0.68,
  elem/.style={circle, fill=black, inner sep=1.6pt},
  lab/.style={font=\small, fill=white, inner sep=1pt},
  edge/.style={draw, thick}
]

\node[elem] (B) at ( 0.0, 5.0) {};   
\node[elem] (L) at (-2.5, 2.5) {};   
\node[elem] (U) at ( 2.5, 2.5) {};   
\node[elem] (C) at ( 0.0, 0.0) {};   

\node[elem] (E) at (0,2.5) {};       

\node[elem] (LE) at (-1.25, 3.75) {}; 
\node[elem] (UE) at ( 1.25, 3.75) {}; 

\node[elem] (D)  at (0,-0.9) {};
\node[elem] (A2) at (0,-1.8) {};
\node[elem] (A1) at (0,-2.7) {};
\node[elem] (T)  at (0,-3.6) {};

\node[lab, above=4pt] at (B) {\(\B_3\)};
\node[lab, left=4pt]  at (L) {\(\mathsf{V}(I(L_3))\)};
\node[lab, right=4pt] at (U) {\(\mathsf{V}(I(U_3))\)};
\node[lab, right=6pt] at (C) {\(\mathsf{V}(I(C_3))\)};
\node[lab, below right=2pt] at (E) {\(\mathsf{V}(I(E_3))\)};

\node[lab, above left=2pt]  at (LE) {\(\mathsf{V}(I(L_3))\vee\mathsf{V}(I(E_3))\)};
\node[lab, above right=2pt] at (UE) {\(\mathsf{V}(I(U_3))\vee\mathsf{V}(I(E_3))\)};

\node[lab, right=3pt] at (D)  {\(\mathsf{V}(I(E_2))\)};
\node[lab, right=3pt] at (A2) {\(\mathsf{V}(I(C_2))\)};
\node[lab, right=3pt] at (A1) {\(\mathsf{V}(I(C_1))\)};
\node[lab, right=3pt] at (T)  {\(\mathbf{T}\)};

\draw[edge] (T)--(A1)--(A2)--(D)--(C);

\draw[edge] (C)--(L);
\draw[edge] (C)--(E);
\draw[edge] (C)--(U);

\draw[edge] (L)--(LE);
\draw[edge] (E)--(LE);
\draw[edge] (U)--(UE);
\draw[edge] (E)--(UE);

\draw[edge] (LE)--(B);
\draw[edge] (UE)--(B);

\end{tikzpicture}
\caption{The subvariety lattice \(\mathcal{L}(\B_3)\). }
\label{fig:lat-b3}
\end{figure}

\begin{remark}
Let \(\V\) be a variety, \(\U\) a subvariety of \(\V\), and \(\sigma\)
an identity satisfied by \(\U\) but not by \(\V\). Then it is easy to
verify that \(\U\) is the subvariety of \(\V\) defined by \(\sigma\).
Using this basic fact, together with Lemma~\ref{lem:identities}, the
explicit finite bases for the subvarieties of \(\mathcal B_3\) are as
follows.
\begin{itemize}
\item \(\mathsf{V}(\mathcal B_3)=\Mod(\Delta_{\infty}, N_3)\).
\item \(\mathsf{V}(I(L_3),I(E_3)) = \Mod(\Delta_{\infty}, N_3, \beta)\).
\item \(\mathsf{V}(I(U_3),I(E_3)) = \Mod(\Delta_{\infty}, N_3, \alpha)\).
\item \(\mathsf{V}(I(L_3)) = \Mod(\Delta_{\infty}, N_3, \gamma)\).
\item \(\mathsf{V}(I(U_3)) = \Mod(\Delta_{\infty}, N_3, \delta)\).
\item \(\mathsf{V}(I(E_3)) = \Mod(\Delta_{\infty}, N_3, \alpha,\beta)\).
\item \(\mathsf{V}(I(C_3)) = \mathcal{V}_3 = \Mod(\Delta_3)\).
\item \(\mathcal{V}(I(E_2)) = \mathcal B_2 = \Mod(\Delta_{\infty},N_2)\).
\item \(\mathcal{V}(I(C_2)) = \mathcal V_2 = \Mod(\Sigma_2)\).
\item \(\mathcal{V}(I(C_1)) = \mathcal V_1 = \Mod(\Delta_1)\).
\item \(\mathbf{T}= \Mod(x\approx y)\).
\end{itemize}
\end{remark}

Since \(\B_3\) is a maximal subvariety of \(\V_4\), we immediately
obtain the following corollary.

\begin{corollary}\label{coro:crossv4}
The variety \(\V_4\) is a Cross variety, and it contains exactly \(12\)
subvarieties.
\end{corollary}

\section{A continuum many subvarieties of height four}\label{sec:independence}
In this section, we prove that \(\mathcal B_k\) contains continuum many
subvarieties for each \(k\geq 4\). We adopt a semantic approach, namely,
we construct a countably infinite family of independent algebras.
A family \((A_i)_{i\in I}\) of algebras is said to be \emph{independent}
if there is a family \((\sigma_i)_{i\in I}\) of identities such that,
for all \(i,j\in I\), $A_i\vDash\sigma_j$ if and only if $i\neq j$.

The construction of independent algebras in this section relies on the
fact that no cycle graph embeds into another. For each cycle of length \(n\geq 5\), we shall construct a
corresponding bounded poset \(P_n\) of height \(4\), and then derive
the associated ai-semiring \(I(P_n)\) together with the ai-semiring
term \(\bt_{P_n}\).

In the following, every graph under consideration is a simple undirected
graph in which each vertex lies in at least one edge and at least one
nonedge. For each such graph \(G=(V(G),E(G))\), write
\[
 D(G)=\binom{V(G)}2\setminus E(G),
\]
and define a disjoint set
\[
 P(G)=\{\bot,\top\}\sqcup\{\ell_e:e\in E(G)\}
       \sqcup\{m_v:v\in V(G)\}\sqcup\{u_d:d\in D(G)\}.
\]
Its strict order is generated by
\begin{equation}\label{eq:graph-order}
 \bot<\ell_e<m_v<u_d<\top,\quad v\in e\cap d.
\end{equation}

For each \(n\geq 5\), let \(Z_n\) denote the cycle of length $n$ with vertex set
\[
 V(Z_n)=\{0,1,2,\ldots,n-1\}
\]
and edge set
\[
 E(Z_n)=\{\{i,i+1\}\mid 0\leq i\leq n-1\},
\]
where the indices are taken modulo \(n\), so that \(n=0\).
In \(Z_n\), every vertex has exactly two neighbors and \(n-3\) distinct
non-neighbors; hence all hypotheses of the construction are satisfied.
Set \(P_n=P(Z_n)\).
Figures~\ref{fig:five-cycle} and~\ref{fig:five-cycle-hasse} show the
\(5\)-cycle \(Z_5\) together with its nonedges, and the Hasse diagram of
the poset \(P_5\).

\begin{figure}[htbp]
\centering
\begin{tikzpicture}[
  scale=0.8,
  vertex/.style={circle, fill=black, inner sep=1.6pt},
  edge/.style={thick},
  nonedge/.style={thick, dashed},
  lab/.style={font=\small}
]

\node[vertex] (v0) at (90:2)   {};
\node[vertex] (v1) at (162:2)  {};
\node[vertex] (v2) at (234:2)  {};
\node[vertex] (v3) at (306:2)  {};
\node[vertex] (v4) at (18:2)   {};

\node[lab] at (90:2.3)   {\(0\)};
\node[lab] at (162:2.3)  {\(1\)};
\node[lab] at (234:2.3)  {\(2\)};
\node[lab] at (306:2.3)  {\(3\)};
\node[lab] at (18:2.3)   {\(4\)};

\draw[edge] (v0) -- (v1);
\draw[edge] (v1) -- (v2);
\draw[edge] (v2) -- (v3);
\draw[edge] (v3) -- (v4);
\draw[edge] (v4) -- (v0);

\draw[nonedge] (v0) -- (v2);
\draw[nonedge] (v0) -- (v3);
\draw[nonedge] (v1) -- (v3);
\draw[nonedge] (v1) -- (v4);
\draw[nonedge] (v2) -- (v4);

\end{tikzpicture}
\caption{A \(5\)-cycle \(Z_5\) and its nonedges.}
\label{fig:five-cycle}
\end{figure}
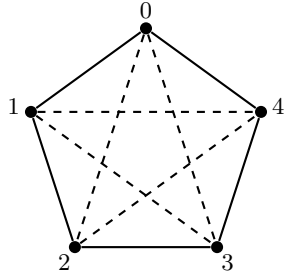

\begin{figure}[htbp]
\centering
\begin{tikzpicture}[
  scale=0.53,
  elem/.style={circle, fill=black, inner sep=1.4pt},
  lab/.style={font=\footnotesize},
  edge/.style={thick}
]

\node[elem] (bot) at (0,0) {};
\node[elem] (top) at (0,8) {};

\node[lab] at (0,-0.5) {\(\bot\)};
\node[lab] at (0,8.5)  {\(\top\)};

\node[elem] (l01) at (-5.0,1.5) {};
\node[elem] (l12) at (-2.5,1.5) {};
\node[elem] (l23) at (0.0,1.5)  {};
\node[elem] (l34) at (2.5,1.5)  {};
\node[elem] (l40) at (5.0,1.5)  {};

\node[lab] at (-5.7,1.5) {\(\ell_{01}\)};
\node[lab] at (-3.0,1.3) {\(\ell_{12}\)};
\node[lab] at (0.7,1.2)  {\(\ell_{23}\)};
\node[lab] at (3.0,1.3)  {\(\ell_{34}\)};
\node[lab] at (5.7,1.5)  {\(\ell_{40}\)};

\node[elem] (m0) at (-5.0,3.5) {};
\node[elem] (m1) at (-2.5,3.5) {};
\node[elem] (m2) at (0.0,3.5)  {};
\node[elem] (m3) at (2.5,3.5)  {};
\node[elem] (m4) at (5.0,3.5)  {};

\node[lab] at (-5.7,3.5) {\(m_0\)};
\node[lab] at (-2.5,3.9) {\(m_1\)};
\node[lab] at (0.0,3.9)  {\(m_2\)};
\node[lab] at (2.5,3.9)  {\(m_3\)};
\node[lab] at (5.7,3.5)  {\(m_4\)};

\node[elem] (u02) at (-5.0,5.5) {};
\node[elem] (u03) at (-2.5,5.5) {};
\node[elem] (u13) at (0.0,5.5)  {};
\node[elem] (u14) at (2.5,5.5)  {};
\node[elem] (u24) at (5.0,5.5)  {};

\node[lab] at (-5.7,5.8) {\(u_{02}\)};
\node[lab] at (-2.9,5.9) {\(u_{03}\)};
\node[lab] at (0.7,5.9)  {\(u_{13}\)};
\node[lab] at (2.9,5.9)  {\(u_{14}\)};
\node[lab] at (5.7,5.8)  {\(u_{24}\)};

\draw[edge] (bot) -- (l01);
\draw[edge] (bot) -- (l12);
\draw[edge] (bot) -- (l23);
\draw[edge] (bot) -- (l34);
\draw[edge] (bot) -- (l40);

\draw[edge] (l01) -- (m0);
\draw[edge] (l01) -- (m1);
\draw[edge] (l12) -- (m1);
\draw[edge] (l12) -- (m2);
\draw[edge] (l23) -- (m2);
\draw[edge] (l23) -- (m3);
\draw[edge] (l34) -- (m3);
\draw[edge] (l34) -- (m4);
\draw[edge] (l40) -- (m4);
\draw[edge] (l40) -- (m0);

\draw[edge] (m0) -- (u02);
\draw[edge] (m0) -- (u03);
\draw[edge] (m1) -- (u13);
\draw[edge] (m1) -- (u14);
\draw[edge] (m2) -- (u02);
\draw[edge] (m2) -- (u24);
\draw[edge] (m3) -- (u03);
\draw[edge] (m3) -- (u13);
\draw[edge] (m4) -- (u14);
\draw[edge] (m4) -- (u24);

\draw[edge] (u02) -- (top);
\draw[edge] (u03) -- (top);
\draw[edge] (u13) -- (top);
\draw[edge] (u14) -- (top);
\draw[edge] (u24) -- (top);

\end{tikzpicture}
\caption{Hasse diagram of the poset \(P_5\).}
\label{fig:five-cycle-hasse}
\end{figure}
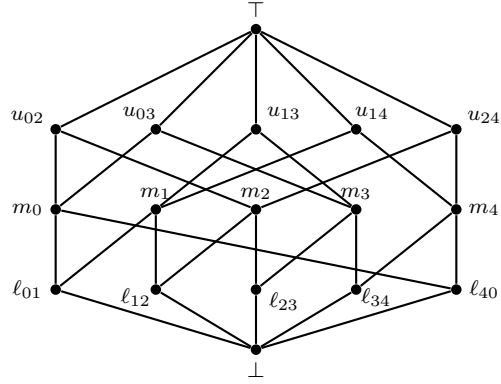

Recall that a \emph{module} of a graph is a set $M$ of
vertices such that each vertex outside $M$ is adjacent
either to every point of $M$ or to no point of $M$.
For a map $f\colon X\to Y$, the set $f^{-1}(y)$ is called a \emph{fiber} of $X$.

\begin{lemma}\label{lem:fibers}
A strictly order preserving map $F\colon P(G)\to P(H)$ induces a vertex map $f$
through $F(m_v)=m_{f(v)}$. Whenever $f(v)\ne f(w)$,
\[
 \{v,w\}\in E(G)\quad\Longleftrightarrow\quad
 \{f(v),f(w)\}\in E(H).
\]
Moreover, the fiber $f^{-1}(a)$ is a module of $G$ for each $a\in V(H)$.
\end{lemma}
\begin{proof}
For each $v\in V(G)$, by hypothesis there exist $e\in E(G)$
and $d\in D(G)$ such that $v\in e\cap d$. Thus
\[
 \bot<\ell_e<m_v<u_d<\top.
\]
Since $F$ is strictly order preserving, it follows that
\[
 F(\bot)<F(\ell_e)<F(m_v)<F(u_d)<F(\top).
\]
By the definition of $P(H)$, this forces
\[
 F(\ell_e)=\ell_{e'},\qquad F(m_v)=m_{v'},\qquad F(u_d)=u_{d'}
\]
for some $e'\in E(H)$, $v'\in V(H)$, $d'\in D(H)$ with
$v'\in e'\cap d'$. Since the elements $m_{v'}$ are distinct
for distinct $v'\in V(H)$, the vertex $v'$ is uniquely
determined by $F(m_v)$. Hence $f(v):=v'$ is well-defined.

Now suppose $f(v)\ne f(w)$.
If $e=\{v,w\}\in E(G)$, then $\ell_e<m_v$ and $\ell_e<m_w$.
Applying the strict order preservation of $F$, we get
\[
 \ell_{e'}=F(\ell_e)<F(m_v)=m_{f(v)},\qquad
 \ell_{e'}=F(\ell_e)<F(m_w)=m_{f(w)}.
\]
Thus $\{f(v),f(w)\}= e'\in E(H)$.

Conversely, if $d=\{v,w\}\in D(G)$, then $m_v<u_d$ and
$m_w<u_d$. Hence
\[
 m_{f(v)}=F(m_v)<F(u_d)=u_{d'},\qquad
 m_{f(w)}=F(m_w)<F(u_d)=u_{d'}.
\]
Thus $\{f(v),f(w)\} = d'\in D(H)$.
This proves
\[
 \{v,w\}\in E(G)\quad\Longleftrightarrow\quad
 \{f(v),f(w)\}\in E(H).
\]

Finally, fix a vertex $a\in V(H)$. If $f^{-1}(a)\in \{\varnothing,V(G)\}$, then it is trivially a module.
So assume $f^{-1}(a)$ is a nonempty proper subset of $V(G)$.
Fix an arbitrary vertex $z\in V(G)\backslash f^{-1}(a)$.
For every $s\in f^{-1}(a)$, we have $f(s)=a\ne f(z)$, so the
equivalence above gives
\[
 z\text{ is adjacent to }s
 \quad\iff\quad
 f(z)\text{ is adjacent to }a.
\]
Therefore $z$ is adjacent either to every vertex $s$
in $f^{-1}(a)$ or to none of them. Hence $f^{-1}(a)$ is a module of $G$ for each $a\in V(H)$.
\end{proof}

\begin{lemma}\label{lem:cycle-modules}
For each $n\geq 5$, every module of $Z_n$ having at least two vertices
is the whole vertex set.
\end{lemma}
\begin{proof}
Suppose a module $M$ is proper and has at least two vertices.
Connectedness gives an edge from $M$ to its complement;
write $z$ for its endpoint outside $M$. The module property
makes $z$ adjacent to every point of $M$. Its degree is two,
so $M=\{z-1,z+1\}$. The vertex $z-2$ is outside this set
for $n\geq5$. It is adjacent to $z-1$ and is not adjacent
to $z+1$, since a difference of three is neither $1$ nor
$-1$ modulo $n\geq5$. This contradicts the module property.
\end{proof}

\begin{lemma}\label{lem:cycle-maps}
For any $m,n\geq 5$, a strictly order preserving map $P_n\to P_m$ is either constant on the vertex level or injective on that level.
The injective case forces $m=n$. In particular, for $m\neq n$ every such map collapses the whole vertex level.
\end{lemma}

\begin{proof}
If the vertex map is nonconstant, its nonempty fibers are proper modules by Lemma~\ref{lem:fibers}.
They are singletons by Lemma~\ref{lem:cycle-modules}, so the map is injective.
Lemma~\ref{lem:fibers} then shows that $f$ is a graph embedding
and that its image induces a subgraph isomorphic to $Z_n$.
Hence $Z_m$ contains a cycle of length $n$.
Since the only cycle in $Z_m$ is $Z_m$ itself, we get $n=m$.
\end{proof}

Next, we introduce the ai-semiring terms induced by the cycle graph
\(Z_n\). First, write $V_n=V(Z_n)$, $E_n=E(Z_n)$ and $D_n=D(Z_n)$.
For each triple \((v,e,d)\in V_n\times E_n\times D_n\) with
\(v\in e\cap d\), We introduce pairwise distinct variables
\[
a_e,\quad b_{e,v},\quad c_{v,d},\quad d_d,\quad
\]
where $(v,e,d)\in V_n\times E_n\times D_n,~ v\in e\cap d$.
These four families name the four kinds of
cover relations in \eqref{eq:graph-order}. Set
\begin{equation}\label{eq:tn}
 \bt_{P_n}=
 \sum_{\substack{(v,e,d)\in V_n\times E_n\times D_n\\ v\in e\cap d}}
 a_e b_{e,v}c_{v,d}d_d,
\end{equation}
where each vertex lies in exactly two edges and \(n-3\) nonedges, so it
contributes \(2(n-3)\) summands. Hence \(\bt_{P_n}\) has exactly
\(2n(n-3)\) summands.

Take $e_* =\{0,1\}\in E_n$ and $d_* =\{1,3\}\in D_n$, and define an identity
\begin{equation}\label{eq:sigma}
 \sigma_n:~\bt_{P_n}\approx \bt_{P_n}+\bw_n.
\end{equation}
where $\bw_n=a_{e_*}b_{e_*,0}c_{1,d_*}d_{d_*}\notin \bt_{P_n}$.

\begin{lemma}\label{lem:recovery}
Let $Q$ be a poset and let $\varphi\colon c(\bt_{P_n})\to I(Q)$ be an arbitrary assignment.
If $\varphi(\bt_{P_n})=(a,b)\ne0$, then there is a strictly order preserving map $f\colon P_n\to Q$ with
$f(\bot)=a$, $f(\top)=b$, and, for every triple $(v,e,d)\in V_n\times E_n\times D_n$ with $v\in e\cap d$,
\begin{equation}\label{formula4}
\begin{aligned}
 \varphi(a_e)=(f(\bot),f(\ell_e)),\qquad
 \varphi(b_{e,v})=(f(\ell_e),f(m_v)),\\
 \varphi(c_{v,d})=(f(m_v),f(u_d)),\qquad
 \varphi(d_d)=(f(u_d),f(\top)).
\end{aligned}
\end{equation}
Conversely, every strictly order preserving map $f\colon P_n\to Q$ defines an
assignment $\varphi$ in $I(Q)$ by these four formulas~\eqref{formula4}, and
then $\varphi(\bt_{P_n})=(f(\bot),f(\top))\ne0$.
\end{lemma}

\begin{proof}
By the flat addition law, a nonempty finite sum is nonzero
if and only if all its summands are equal to the same nonzero element.
Hence, there exists $a,b\in Q$ with $a<b$ such that
\[
 \varphi(a_e b_{e,v}c_{v,d}d_d)=(a,b)
\]
for every triple $(v,e,d)\in V_n\times E_n\times D_n$ with $v\in e\cap d$. Since every
variable occurs in at least one such product, every variable
has nonzero value.

Define, for each $e\in E_n$,
\[
 \varphi(a_e)=(a,L_e).
\]
This is well defined because $a$ is the left endpoint of
$\varphi(a_e)$: the factor $a_e$ is the first factor of a
product of value $(a,b)$, and interval multiplication forces
endpoints to match.

Similarly, for each $d\in D_n$,
\[
 \varphi(d_d)=(U_d,b).
\]

Fix $v\in V(Z_n)$ and choose $e\in E_n$, $d\in D_n$ with
$v\in e\cap d$. Since
\[
 \varphi(a_e b_{e,v}c_{v,d}d_d)=(a,b),
\]
the four factors compose, so
\[
 \varphi(b_{e,v})=(L_e,M_{v,e,d}),\qquad
 \varphi(c_{v,d})=(M_{v,e,d},U_d)
\]
for some $M_{v,e,d}\in Q$.

We claim that $M_{v,e,d}$ is independent of $e$ and $d$.
Fix $d\in D_n$ with $v\in d$. For any two edges
$e,e'\in E_n$ containing $v$, both
$\varphi(b_{e,v})$ and $\varphi(b_{e',v})$ start at
$a$, have terminal endpoint in $\varphi(c_{v,d})$, and
the product with $\varphi(c_{v,d})$ is nonzero. Hence
their terminal endpoints coincide. Thus $M_{v,e,d}$
is independent of $e$. The same argument with $e$ fixed
shows independence of $d$. Therefore there is a unique
$M_v\in Q$ such that
\[
 \varphi(b_{e,v})=(L_e,M_v),\qquad
 \varphi(c_{v,d})=(M_v,U_d)
\]
for all incident $e,d$.

Now define
\[
 f(\bot)=a,\qquad f(\ell_e)=L_e,\qquad f(m_v)=M_v,
\]
\[
 f(u_d)=U_d,\qquad f(\top)=b.
\]
If $\bot<\ell_e<m_v<u_d<\top$ is a generating comparison
in $P_n$, then the corresponding variable values are strict
intervals:
\[
 (a,L_e),\quad (L_e,M_v),\quad (M_v,U_d),\quad (U_d,b),
\]
so $a<L_e<M_v<U_d<b$.
Hence \(f\) is strictly order preserving on generating comparisons, and therefore on all of \(P_n\) by transitivity.
The displayed formulas for $\varphi$ follow by construction.

Conversely, let $f\colon P_n\to Q$ be a strictly order preserving map. Define
$\varphi$ on the variables by
\[
 \varphi(a_e)=(f(\bot),f(\ell_e)),\qquad
 \varphi(b_{e,v})=(f(\ell_e),f(m_v)),
\]
\[
 \varphi(c_{v,d})=(f(m_v),f(u_d)),\qquad
 \varphi(d_d)=(f(u_d),f(\top)).
\]
For every triple $(v,e,d)$ with $v\in e\cap d$, strict order preservation
of $f$ gives
\[
 f(\bot)<f(\ell_e)<f(m_v)<f(u_d)<f(\top),
\]
so the four factors compose and
\[
 \varphi(a_e b_{e,v}c_{v,d}d_d)
 =(f(\bot),f(\top))\ne0.
\]
Thus every summand of $\varphi(\bt_{P_n})$ equals
$(f(\bot),f(\top))$, and flat addition gives
\[
 \varphi(\bt_{P_n})=(f(\bot),f(\top))\ne0.
\]
Neither direction uses injectivity of $f$, distinctness of
the variable values, or finiteness of $Q$.
\end{proof}

\begin{lemma}\label{lem:criterion}
For every poset \(Q\) and every \(n\geq 5\),
the interval semiring $I(Q)$ satisfies $\sigma_n$ if and only if
 $f(m_0)=f(m_1)$ for every strictly order preserving map $f\colon P_n\to Q$.
\end{lemma}

\begin{proof}
Let \(\varphi\colon c(\bt_{P_n})\to I(Q)\) be an arbitrary assignment.
If \(\varphi(\bt_{P_n})=0\), then
\[
 \varphi(\bt_{P_n})=0=\varphi(\bt_{P_n}+\bw_n),
\]
since \(0+s=0\) for all \(s\in I(Q)\). Therefore, \(I(Q)\vDash\sigma_n\) if and only if
\(\varphi(\bt_{P_n}) = \varphi(\bw_n)\) for every assignment \(\varphi\) with
\(\varphi(\bt_{P_n})\neq 0\).

Now suppose that \(\varphi(\bt_{P_n})\neq 0\).
By Lemma~\ref{lem:recovery}, there is a strictly order preserving map \(f\colon P_n\to Q\)
induced by \(\varphi\) such that
\[
 \varphi(\bt_{P_n})=(f(\bot),f(\top)),
\]
and such that
\[
 \varphi(a_{e_*}b_{e_*,0}) = (f(\bot),f(m_0)),
 \qquad
 \varphi(c_{1,d_*}d_{d_*}) = (f(m_1),f(\top)).
\]

Consequently, recall that \(\bw_n = a_{e_*}b_{e_*,0}c_{1,d_*}d_{d_*}\), we have
\begin{equation}\label{eqwn=tpn}
 \varphi(\bw_n)
 =(f(\bot),f(m_0))\cdot(f(m_1),f(\top))
 =
 \begin{cases}
  \varphi(\bt_{P_n}), & f(m_0)=f(m_1),\\
  0, & f(m_0)\ne f(m_1).
 \end{cases}
\end{equation}
Conversely, by Lemma~\ref{lem:recovery} again, every strictly order preserving map \(f\colon P_n\to Q\)
induces an assignment \(\varphi\) such that
\(\varphi(\bt_{P_n}) = (f(\bot),f(\top))\), and~\eqref{eqwn=tpn} holds.
In summary, \(I(Q)\vDash\sigma_n\) if and only if
\(f(m_0)=f(m_1)\) for every strictly order preserving map \(f\colon P_n\to Q\).
\end{proof}

\begin{lemma}\label{lem:independence}
For all \(m,n\geq 5\), $I(P_m)\vDash\sigma_n$ if and only if $m\ne n$.
\end{lemma}

\begin{proof}
Suppose first that \(m=n\). Then the identity map
\(f\colon P_n\to P_n\) is clearly strictly order preserving. Since
\(f(m_0)\neq f(m_1)\), it follows from Lemma~\ref{lem:criterion} that
\(P_n\nvDash\sigma_n\).

Now suppose that \(m\ne n\). By Lemma~\ref{lem:cycle-maps}, every strictly order preserving map \(f\colon P_n\to P_m\) collapses the vertex level; in particular,
\(f(m_0)=f(m_1)\). Applying Lemma~\ref{lem:criterion} once more, we
obtain \(I(P_m)\vDash\sigma_n\). This completes the proof.
\end{proof}

Recall from Proposition~\ref{prop:seven} that, for every bounded poset
\(P\) of height at most \(3\), the variety \(\mathsf V(I(P))\) is equal
to the variety generated by the interval semirings of those bounded
subposets of \(P\) that are isomorphic to one of seven fixed bounded
posets. The following proposition shows that this conclusion does not
extend to height \(4\).

\begin{proposition}\label{prop:no-bound}
There is no positive integer $k$ such that every
bounded poset $P$ of height at most $4$ satisfies
\[
\Var(I(P)) = \mathsf{V}\{I(Q)\mid Q\text{ is a subposet of $P$ with $|Q|\leq k$}\}.
\]
\end{proposition}
\begin{proof}
Suppose on the contrary that there is an integer $k\geq 1$ such that
\[
\Var(I(P)) = \mathsf{V}\{I(Q)\mid Q\text{ is a subposet of $P$ with $|Q|\leq k$}\}
\]
for each bounded poset $P$ of height at most four.
Take $n=k+5$ and $P=P_n$.
For each subposet $Q\subseteq P_n$ with at most
$k$ points, compose a strictly order preserving map $P_n\to Q$ with
the inclusion $Q\hookrightarrow P_n$. By
Lemma~\ref{lem:cycle-maps} the composite is constant
or injective on the $n$ vertex points. Injectivity would
require $|Q|\geq n>k$, so only the constant case is
possible. The inclusion is injective, and hence every
strictly order preserving map from $P_n$ into $Q$ collapses $m_0$ and $m_1$.
Lemma~\ref{lem:criterion} gives $I(Q)\vDash\sigma_n$.
 Its proposed
member $I(P_n)$ fails that identity by
Lemma~\ref{lem:independence}. This disproves the
bound $k$.
\end{proof}

\begin{proposition}\label{prop:k=4}
The interval \([\mathcal V_4,\mathcal B_4]\) contains continuum many subvarieties.
\end{proposition}

\begin{proof}
Let \(J=\{5,6,7,\dots\}\). For each nonempty subset \(X\) of \(J\), define the variety
\[
\mathcal V_X=\mathsf V(I(P_n)\mid n\in X).
\]
Since \(\htp(P_n)=4\), Lemma~\ref{lem:sibk} implies that every
\(I(P_n)\) contains \(I(C_4)\) as a subsemiring and lies in the variety
\(\mathcal B_4\).
Thus \(\mathcal V_X\) lies in the interval \([\mathcal V_4,\mathcal B_4]\) for every nonempty subset \(X\) of \(J\).

Now suppose that \(X_1\) and \(X_2\) are distinct nonempty subsets of \(J\).
Without loss of generality, assume that \(n\in X_1\setminus X_2\).
Then, by Lemma~\ref{lem:independence}, \(\mathcal V_{X_2}\) satisfies \(\sigma_n\), whereas \(\mathcal V_{X_1}\) does not.
This shows that distinct nonempty subsets of the countably infinite set \(J\) determine distinct subvarieties in the interval \([\mathcal V_4,\mathcal B_4]\).
Hence the interval \([\mathcal V_4,\mathcal B_4]\) contains continuum many subvarieties.
\end{proof}

\begin{proposition}\label{prop:intervalk}
For every \(k\geq 4\), the interval \([\V_k,\B_k]\) contains continuum many
subvarieties.
\end{proposition}

\begin{proof}
The proof is entirely analogous in the case $k=4$. By Proposition~\ref{prop:k=4}, assume \(k\geq 5\) and put
\(r=k-5\). Let \(P_n^{(k)}\) be the ordinal sum
\[
P_n^{(k)}=C_r\oplus P_n,
\]
where \(C_r = \{0<1<\cdots<r\}\) contributes \(r+1 = k-4\) additional steps below \(P_n\) (note that \(r<\bot\)), so that \(P_n^{(k)}\) has height \(k\).
Define the term
\[
 Q_r=x_1x_2\cdots x_r
\]
and the identity
\[
 \sigma_n^{(k)}:\quad
 Q_r\,\bt_{P_n}\approx Q_r(\bt_{P_n}+\bw_n).
\]

Every strictly order preserving map \(P_n^{(k)}\to P_m^{(k)}\) must preserve the lowest
\(k\) points, hence restricts to a strictly order preserving map \(P_n\to P_m\) on the
copy of \(P_n\) sitting above \(C_r\). Therefore the results of this
section apply verbatim to the algebras \(I(P_n^{(k)})\). In
particular, for every nonempty subset \(X\subseteq J=\{k+1,k+2,\dots\}\),
the varieties
\[
 \V_X^{(k)}=\mathsf V\bigl(I(P_n^{(k)})\mid n\in X\bigr)
\]
are pairwise distinct and all lie in \([\V_k,\B_k]\). Hence
\([\V_k,\B_k]\) contains continuum many subvarieties.
\end{proof}

For each \(k\geq 1\), it is easy to see that
\[
 \V_k\subseteq\B_k\subseteq\V_{k+1}\subseteq\B_{k+1}\subseteq\V_{\infty}.
\]
Combining this with Proposition~\ref{prop:intervalk}, which shows that the
interval \([\V_k,\B_k]\) contains \(2^{\aleph_0}\) many subvarieties for every $k\geq 4$,
we immediately obtain the following consequence.

\begin{theorem}\label{cor:intervals}
For every \(k\geq 4\), each of the following intervals
\[
 [\V_k,\B_k],\qquad [\V_k,\V_{k+1}],\qquad [\B_{k-1},\B_{k}]
\]
contains \(2^{\aleph_0}\) many subvarieties. For \(k\leq 3\), the
above intervals are all finite.
\end{theorem}

In particular, we obtain infinitely many finitely based finite semirings whose generated varieties have continuum many subvarieties.
See the following proposition.
\begin{proposition}
For each \(k\geq 4\), the interval semirings \(I(C_{k+1})\) and
\(I(Q_k)\) are finitely based finite semirings whose varieties contain
continuum many subvarieties.
\end{proposition}

\begin{proposition}
The flat semiring \(B_0\) is a finitely based finite semiring whose
variety contains continuum many subvarieties.
\end{proposition}

\begin{remark}
This solves the problem raised in \cite[Remark 5.5]{g}.
\end{remark}

\section{The limit subvariety of $\V_{\infty}$}\label{sec:limit}

For each \(k\geq 2\), let \(E_k\) be the bounded poset formed by two \(k\)-step chains sharing
only their least and greatest elements. Its only strict order relations are
\[
\bot<a_1<a_2<\cdots<a_{k-1}<\top,
\]
\[
\bot<b_1<b_2<\cdots<b_{k-1}<\top,
\]
with no comparisons between the two chains. See
Figure~\ref{fig:seven-posets} for the cases \(E_2\) and \(E_3\).

As can be seen from
Figure~\ref{fig:lat-b3}, \(\V_k\) is a maximal subvariety of
\(\Var(E_k)\) for \(k=2,3\). In the following, we shall prove that
\(\V_k\) is a maximal subvariety of \(\Var(I(E_k))\), and that
$I(E_k)$ is nonfinitely based for each $k\geq 4$.

We first treat the case \(k=4\) by a semantic method. The strategy is as
follows. For each positive integer \(n\), we construct an algebra
\(A_n\) such that every \(n\)-generated subalgebra of \(A_n\) lies in
\(\Var(I(E_4))\), and hence \(A_n\) satisfies every \(n\)-variable
identity of \(\Var(I(E_4))\), while \(A_n\) itself does not lie in
\(\Var(I(E_4))\). Thus \(A_n\) fails some identity of
\(\Var(I(E_4))\), which shows that the set of all \(n\)-variable
identities of \(\Var(I(E_4))\) does not form a basis of
\(\Var(I(E_4))\). In particular, \(\Var(I(E_4))\) is nonfinitely based.

Concretely, for each \(n\geq 2\), construct \(F_n\) by adjoining to the
two anchor chains of \(E_4\) upper fence interior points
\(u_1,\dots,u_{n-1}\) and lower fence interior points
\(v_1,\dots,v_n\), subject to the strict order relations
\[
a_3>v_1<u_1>v_2<u_2>\cdots>v_n<b_3.
\]
See Figure~\ref{fig:F3} for the Hasse diagram of \(F_3\);
for the sake of clarity, the order relations
\(\bot<u_i,v_j<\top\) are indicated by gray dashed lines.

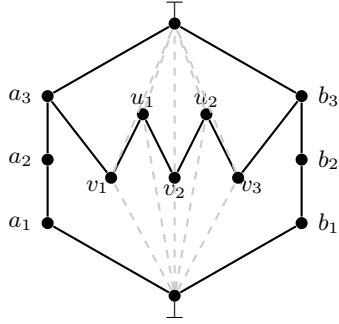
\begin{figure}[htbp]
\centering
\begin{tikzpicture}[
  scale=0.6,
  every node/.style={font=\small},
  elem/.style={circle, fill=black, inner sep=1.6pt},
  lab/.style={font=\small},
  edge/.style={thick}
]

\node[elem] (B) at (0,-0.8) {};
\node[lab] at (0,-1.1) {\(\bot\)};

\node[elem] (T) at (0,5.2) {};
\node[lab] at (0,5.5) {\(\top\)};

\node[elem] (a1) at (-2.8,0.8) {};
\node[elem] (a2) at (-2.8,2.2) {};
\node[elem] (a3) at (-2.8,3.6) {};

\node[lab] at (-3.4,0.8) {\(a_1\)};
\node[lab] at (-3.4,2.2) {\(a_2\)};
\node[lab] at (-3.4,3.6) {\(a_3\)};

\node[elem] (b1) at (2.8,0.8) {};
\node[elem] (b2) at (2.8,2.2) {};
\node[elem] (b3) at (2.8,3.6) {};

\node[lab] at (3.4,0.8) {\(b_1\)};
\node[lab] at (3.4,2.2) {\(b_2\)};
\node[lab] at (3.4,3.6) {\(b_3\)};

\node[elem] (v1) at (-1.4,1.8) {};
\node[elem] (v2) at (0.0,1.8) {};
\node[elem] (v3) at (1.4,1.8) {};

\node[lab] at (-1.7,1.6) {\(v_1\)};
\node[lab] at (0.0,1.5) {\(v_2\)};
\node[lab] at (1.7,1.6) {\(v_3\)};

\node[elem] (u1) at (-0.7,3.2) {};
\node[elem] (u2) at (0.7,3.2) {};

\node[lab] at (-0.7,3.5) {\(u_1\)};
\node[lab] at (0.7,3.5) {\(u_2\)};

\draw[edge] (B) -- (a1);
\draw[edge] (B) -- (b1);

\draw[edge] (a1) -- (a2);
\draw[edge] (a2) -- (a3);

\draw[edge] (b1) -- (b2);
\draw[edge] (b2) -- (b3);

\draw[edge] (a3) -- (T);
\draw[edge] (b3) -- (T);

\draw[edge] (a3) -- (v1);
\draw[edge] (v1) -- (u1);
\draw[edge] (u1) -- (v2);
\draw[edge] (v2) -- (u2);
\draw[edge] (u2) -- (v3);
\draw[edge] (v3) -- (b3);

\foreach \x in {v1,v2,v3,u1,u2} {
  \draw[edge, gray!40, dashed] (B) -- (\x);
  \draw[edge, gray!40, dashed] (\x) -- (T);
}

\end{tikzpicture}
\caption{The Hasse diagram of \(F_3\).}
\label{fig:F3}
\end{figure}

For each poset \(P_n\), the corresponding ai-semiring term
\(\bt_{P_n}\) was introduced in \eqref{eq:tn}. Every path from
\(\bot\) to \(\top\) in \(P_n\) consists of exactly four steps, and
this yields the \(4\)-uniform ai-semiring term \(\bt_{P_n}\). Using this
idea, we introduce a general ai-semiring term \(\bt_P\) for an
arbitrary bounded poset \(P\) as follows.

For a bounded poset \(P\), give every cover \(a\prec b\) its own
variable \(x_{ab}\). Let \(\bt_P\) be the sum of the cover-variable
words of all saturated bottom-to-top paths; that is,
\[
\bt_P =
\sum_{\bot\prec a_1\prec a_2\prec \cdots \prec a_n\prec \top}
x_{\bot a_1}x_{a_1a_2}\cdots x_{a_n\top}.
\]

For example,
$x_{\bot \ell_{12}}x_{\ell_{12}m_1}x_{m_1u_{13}}x_{u_{13}\top}$
is a summand of \(\bt_{P_4}\). Also, for each \(k\geq 2\),
\[
\bt_{E_k} =
x_{\bot a_1}x_{a_1a_2}\cdots x_{a_{k-1}\top}
+
x_{\bot b_1}x_{b_1b_2}\cdots x_{b_{k-1}\top},
\]
and
\[
\bt_{Q_k} =
\sum_{c_i\in \{a_i,b_i\}}
x_{\bot c_1} x_{c_1 c_2} \cdots x_{c_{k-1}\top}.
\]

Similarly to the proof of Lemma~\ref{lem:recovery}, we obtain the
following analogous conclusion.

\begin{lemma}\label{lem:strict}
Let \(P\) and \(Q\) be bounded posets. Then a valuation of \(\bt_P\) in \(I(Q)\)
is nonzero if and only if it is induced by a strictly order preserving map \(P\to Q\).
\end{lemma}

For incomparable interior points \(p,q\) in a bounded poset $P$, choose a saturated path from
\(\bot\) to \(p\) and a saturated path from \(q\) to \(\top\), and
concatenate their words to obtain \(\bw_{pq}\). Then we have the following result.

\begin{lemma}\label{lem:characteristic}
For any poset \(Q\), the identity \(\bt_P\approx \bt_P+\bw_{pq}\) holds
in \(I(Q)\) if and only if $f(p) = f(q)$ for every strictly order preserving map \(f\colon P\to Q\).
\end{lemma}

\begin{proof}
Consider the word \(\bw_{pq}\). It is obtained by choosing a
saturated path from \(\bot\) to \(p\), a saturated path from \(q\) to
\(\top\), and concatenating their words.
Let $\varphi$ be a valuation induced by
a strictly order preserving map \(f\), the first path evaluates to
\((f(\bot),f(p))\), and the second to \((f(q),f(\top))\). Therefore
\[
\varphi(\bw_{pq}) =
(f(\bot),f(p))\cdot(f(q),f(\top)).
\]
This product is equal to \((f(\bot),f(\top))\) exactly when
\(f(p)=f(q)\), and is \(0\) otherwise.
If \(\varphi(\bt_P) = 0\), then the identity
\(\bt_P\approx\bt_P+\bw_{pq}\) holds automatically, since \(0\) absorbs
addition. Hence
\[
I(Q)\vDash \bt_P\approx\bt_P+\bw_{pq}
\]
if and only if every strictly order preserving map \(f\colon P\to Q\) satisfies
\(f(p)=f(q)\). This proves the identity criterion.
\end{proof}

\begin{proposition}\label{prop:joint-separation}
Let \(P\) be a bounded poset and let \(Q\) be a poset. Then
\[
I(P)\in\Var(I(Q))
\]
if and only if there exist strictly order preserving maps \(P\to Q\) that jointly separate
the points of \(P\).
\end{proposition}

\begin{proof}
By Lemma~\ref{lem:separation-hsp}, sufficiency is immediate.
It remains to prove necessity.

Suppose that \(I(P)\in\Var(I(Q))\). By Lemma~\ref{lem:height-u}, it
follows that
\[
\htp(P)\leq \htp(Q).
\]
Therefore a strictly order preserving map \(P\to Q\) certainly exists. For example, let
\(\htp(P)=n\), and choose a chain
\[
a_0<a_1<\cdots<a_n
\]
in \(Q\) of length \(n\). The rank function then induces a strictly order preserving map \(f\colon P\to Q\) by
\[
f(v)=a_{\ell(\bot,v)}.
\]

It remains only to prove that the strictly order preserving maps \(P\to Q\) jointly
separate \(P\). Since every comparable pair is separated by every strictly order preserving map, only incomparable pairs matter.

Suppose, on the contrary, that some incomparable pair \(p,q\) is
identified by every strictly order preserving map \(P\to Q\). Then, by the identity
criterion proved above, the identity
\[
\bt_P\approx\bt_P+\bw_{pq}
\]
holds in \(I(Q)\), and hence also in \(\Var(I(Q))\). But this identity
fails in \(I(P)\): under the natural cover-interval valuation
\[
x_{ab}\mapsto (a,b),
\]
every summand of \(\bt_P\) evaluates to \((\bot,\top)\), while
\(\bw_{pq}\) evaluates to
\[
(\bot,p)\cdot(q,\top),
\]
which is \(0\) because \(p\neq q\). This contradicts
\(I(P)\in\Var(I(Q))\). Therefore every pair of distinct points of \(P\)
is separated by some strictly order preserving map \(P\to Q\).
\end{proof}

\begin{lemma}\label{lem:fence}
For every \(n\geq 2\), \(I(F_n)\notin\Var(I(E_4))\), whereas
\[
I(F_n\setminus\{u_s\})\in\Var(I(E_4))
\]
for each \(1\leq s\leq n-1\).
\end{lemma}

\begin{proof}
Let \(f\colon F_n\to E_4\) be an arbitrary strictly order preserving map. Then \(f\) sends
the bounds to the bounds, because both \(F_n\) and \(E_4\) have height
\(4\). Its interior points map to interior points. Comparable interior
points must use the same chain of \(E_4\), since the latter consists of
two chains whose interior points are incomparable. On the other hand,
since \(f\) is strict,
\[
f(a_3)>f(v_1)<f(u_1)>f(v_2)<f(u_2)>\cdots>f(v_n)<f(b_3),
\]
it follows that \(f(a_3)\) and \(f(b_3)\) lie in the same chain of
\(E_4\), and both have rank \(3\). Hence \(f(a_3)=f(b_3)\). This shows
that \(a_3\) and \(b_3\) cannot be separated by any strictly order preserving map \(F_n\to E_4\). By Proposition~\ref{prop:joint-separation}, it follows
that \(I(F_n)\notin\Var(I(E_4))\).

Now let \(1\leq s\leq n-1\), and put
\[
Q_{n,s}=F_n\setminus\{u_s\},
\]
with the induced order. Deleting \(u_s\) leaves two interior components,
each containing exactly one anchor chain. Define a map
\(f\colon Q_{n,s}\to E_4\) by fixing the bounds and the two anchor
chains of length \(4\), that is,
\[
f(a_i)=a_i,\qquad f(b_i)=b_i\qquad (i=1,2,3),
\]
and folding the remaining interior points \(v_i\) and \(u_j\) into the
two chains as follows:
\[
f(v_i)=
\begin{cases}
a_2, & 1\leq i\leq s,\\
b_2, & s+1\leq i\leq n,
\end{cases}
\qquad
f(u_j)=
\begin{cases}
a_3, & 1\leq j\leq s-1,\\
b_3, & s\leq j\leq n-1.
\end{cases}
\]
The map \(f\) is strict and separates all distinct pairs of points in $Q_{n,s}$ except
those of the form \(v_i\) versus \(a_2,b_2\), and \(u_j\) versus
\(a_3,b_3\).

Now define a map \(g\colon Q_{n,s}\to E_4\) which again fixes the points on
the two chains of length \(4\), and sends the remaining interior points
as follows:
\[
g(v_i)=
\begin{cases}
a_1, & 1\leq i\leq s,\\
b_1, & s+1\leq i\leq n,
\end{cases}
\qquad
g(u_j)=
\begin{cases}
a_2, & 1\leq j\leq s-1,\\
b_2, & s\leq j\leq n-1.
\end{cases}
\]
It is easy to verify that \(g\) is also strict.
Using \(f\) and \(g\), one separates all points of \(Q_{n,s}\).
By Proposition~\ref{prop:joint-separation}, it follows that
\(I(Q_{n,s})\in\Var(I(E_4))\), as required.
\end{proof}

\begin{proposition}\label{prop:nfb}
The interval semiring \(I(E_4)\) is nonfinitely based.
\end{proposition}

\begin{proof}
For each positive integer \(n\), consider the interval
semiring \(I(F_{2n+2})\). By Lemma~\ref{lem:fence}, $I(F_{2n+2})\notin\Var(I(E_4))$.
On the other hand, let \(S\) be any subsemiring of \(I(F_{2n+2})\)
generated by at most \(n\) elements. The endpoints of the nonzero
generators form a set of at most \(2n\) points. Since there are \(2n+1\)
interior fence uppers, at least one of them, say \(u_s\), is absent from
this set. Hence all generators, together with zero if present, lie in
the subsemiring $I(F_{2n+2}\setminus\{u_s\})$
of \(I(F_{2n+2})\). By Lemma~\ref{lem:fence}, the generated subalgebra
\(S\) belongs to \(\Var(I(E_4))\).

Thus every \(n\)-generated subalgebra of \(I(F_{2n+2})\) lies in the
variety \(\Var(I(E_4))\), whereas \(I(F_{2n+2})\) itself does
not. Hence no finite set of identities axiomatizes the variety  \(\Var(I(E_4))\), so \(I(E_4)\) is nonfinitely based.
\end{proof}

\begin{theorem}
The interval semiring \(I(E_k)\) is nonfinitely based for each
\(k\geq 4\).
\end{theorem}

\begin{proof}
By Proposition~\ref{prop:nfb}, we may assume that \(k\geq 5\). In both
\(E_4\) and each \(F_n\) constructed above, insert \(r=k-4\) new
points into the chain from \(\bot\) to \(a_1\), and similarly into the
chain from \(\bot\) to \(b_1\), denoting them by \(c_1,\dots,c_r\) and
\(d_1,\dots,d_r\), respectively. Since every strictly order preserving preserves the
ranks of interior points, the key Lemma~\ref{lem:fence} remains valid
when \(E_4\) and \(F_n\) are replaced by \(E_k\) and \(F_n^{(k)}\),
respectively. Arguing as in Proposition~\ref{prop:nfb}, one can then
use \(I(F_{2n+2}^{(k)})\) to prove that \(I(E_k)\) is nonfinitely
based.
\end{proof}

\begin{lemma}\label{lem:iekmaxvk}
For each \(k\geq 4\), the unique maximal proper subvariety of
\(\Var(I(E_k))\) is \(\V_k\).
\end{lemma}

\begin{proof}
 We first show that \(\V_k\) is a
proper subvariety of \(\Var(I(E_k))\). The two full chains of \(E_k\)
violate the identity \(U_k\), so \(I(E_k)\notin\V_k\). On the other
hand, each of these two chains is a copy of \(C_k\), so \(I(C_k)\) is a
subalgebra of \(I(E_k)\). Hence \(\V_k\leq\Var(I(E_k))\), and therefore
\(\V_k<\Var(I(E_k))\).

Next let \(\mathcal U\) be a proper subvariety of \(\Var(I(E_k))\). By
Lemma~\ref{localfinite}, the variety \(\Var(I(E_k))\) is locally finite,
so \(\mathcal U\) is generated by its finite subdirectly irreducible
members. By Proposition~\ref{prop:representation}, each such member is
of the form \(I(P)\) for a finite bounded poset \(P\). Since
\(\mathcal U\leq\Var(I(E_k))\), Lemma~\ref{lem:height-u}(i) gives
\(\htp(P)\leq k\).

We claim that every such \(P\) has at most one \(k\)-step chain.
Suppose, on the contrary, that \(P\) contains two different \(k\)-step
chains. Since $I(P)\in \Var(I(E_k))$, by Proposition~\ref{prop:joint-separation}, there exist strictly order preserving maps \(P\to E_k\) that jointly separate the points of \(P\). Under
every strictly order preserving map \(P\to E_k\), the interior points of a \(k\)-step
chain of \(P\) must map to the interior points of one of the two chains
of \(E_k\), because \(E_k\) has height \(k\) and no comparisons between
the interiors of its two chains. If the interiors of the two \(k\)-step
chains of \(P\) lay in the same connected component of the interior
comparability graph of \(P\), then every strictly order preserving map \(P\to E_k\) would
send both chains to the same chain of \(E_k\). Since the two chains are
different, they have distinct interior points at some common rank, and
those two points would then have the same image under every strictly order preserving,
contradicting joint separation. Hence the interiors of the two \(k\)-step
chains lie in different connected components. In particular, they are
disjoint and have no interior cross comparisons. Together with the
bounds, they form an induced copy of \(E_k\) in \(P\). Therefore
\(I(E_k)\) is a subalgebra of \(I(P)\), so
\(\Var(I(E_k))\leq\mathcal U\), contradicting
\(\mathcal U<\Var(I(E_k))\).

Thus every finite subdirectly irreducible generator \(I(P)\) of
\(\mathcal U\) has at most one \(k\)-step chain. By
Proposition~\ref{prop:bounded-separation}, such an \(I(P)\) belongs to \(\V_k\).
Hence \(\mathcal U\leq\V_k\). This implies that \(\V_k\) is the unique
maximal subvariety of \(\Var(I(E_k))\).
\end{proof}

\begin{theorem}\label{thm:explicit}
The variety \(\Var(I(E_4))\) is a limit variety, and it has exactly
\(13\) subvarieties.
\end{theorem}

\begin{proof}
By Lemma~\ref{lem:iekmaxvk}, \(\V_4\) is the unique maximal proper
subvariety of \(\Var(I(E_4))\). By Corollary~\ref{coro:crossv4},
\(\V_4\) is a Cross variety. Hence every proper subvariety of
\(\Var(I(E_4))\) is finitely based. By Proposition~\ref{prop:nfb},
\(\Var(I(E_4))\) itself is nonfinitely based. Therefore
\(\Var(I(E_4))\) is a limit variety, and it has exactly \(13\)
subvarieties.
\end{proof}

\section{Three finite classes controlling the classification}\label{sec:threeclasses}
For each \(k\geq 4\) and \(1\leq i\leq k-1\), set
\[
\bu_k=x_1x_2\cdots x_k+y_1y_2\cdots y_k,
\]
and define
\[
\tau_{k,i}:\quad
\bu_k\approx\bu_k+x_1\cdots x_i y_{i+1}\cdots y_k,
\qquad
\B_{k,i}=\B_k\cap\Mod(\tau_{k,i}).
\]

To understand the effect of the identity \(\tau_{k,i}\) on the members of
\(\B_k\), we state the following lemma.

\begin{lemma}\label{lem:kstepsame}
For each $k\geq 4$, $1\leq i\leq k-1$, an interval
semiring \(I(P)\) of a bounded poset of height \(k\) satisfies the identity \(\tau_{k,i}\) if and only if any
two $k$-step chains in \(P\) have the same \(i\)th interior point.
\end{lemma}

\begin{proposition}\label{prop:bkiqki}
For each \(k\geq 2\) and \(1\leq i\leq k-1\), let $Q_{k,i}=Q_k\setminus\{b_i\}$.
Then
\[
\B_{k,i}=\Var(I(Q_{k,i})).
\]
\end{proposition}

\begin{proof}
Since \(Q_{k,i}\) is a bounded poset of height \(k\) and has a unique
interior point of rank \(i\), it follows from Lemmas~\ref{lem:sibk} and
\ref{lem:kstepsame} that \(I(Q_{k,i})\) lies in the variety \(\B_k\) and
satisfies the identity \(\tau_i\). Thus
\(I(Q_{k,i})\in\B_{k,i}\).

Conversely, let \(I(P)\) be an interval semiring in \(\B_{k,i}\). If
\(\htp(P)\leq k-1\), then by Lemma~\ref{lem:ipck},
\[
I(P)\in\Var(I(C_k))\subseteq\Var(I(Q_{k,i})),
\]
since \(Q_{k,i}\) contains a \(k\)-step chain.

If \(\htp(P)=k\), then any two \(k\)-step chains in \(P\) have the same
\(i\)-th interior point. By the proof of Proposition~\ref{prop:finite},
the poset \(P\) can be separated by a family of strictly order preserving maps
\(P\to Q_k\). Since \(P\) has a unique point of rank \(i\), all the
strictly order preserving maps occurring in that proof may be regarded as strictly order preserving maps
\(P\to Q_{k,i}\). By Proposition~\ref{prop:joint-separation}, it follows
that \(I(P)\in\Var(I(Q_{k,i}))\).

This proves that \(\B_{k,i}=\Var(I(Q_{k,i}))\).
\end{proof}

\begin{proposition}\label{prop:bkbkibkj}
For each \(k\geq 4\), and $1\leq i<j\leq k-1$, $\B_k=\B_{k,i}\lor\B_{k,j}$.
\end{proposition}

\begin{proof}
Since \(\B_{k,i}\) and \(\B_{k,j}\) are subvarieties of \(\B_k\), it
follows that
\[
\B_{k,i}\lor\B_{k,j}\leq\B_k.
\]

Conversely, for $t\in \{i,j\}$, define a map \(f_t\colon Q_k\to Q_{k,t}\) by \(f(b_t)=a_t\)
and fixing all other points.

It is easy to verify that \(f_i\) and \(f_j\) are strictly order preserving maps that jointly
separate \(Q_k\). Hence, by Lemma~\ref{lem:separation-hsp} and Propositions~\ref{prop:finite} and~\ref{prop:bkiqki},
it follows that
\[
\B_k=\Var(I(Q_k))
\subseteq\Var(I(Q_{k,i}),I(Q_{k,j}))
=\B_{k,i}\lor\B_{k,j}.
\]
Together with the reverse inclusion, this gives
\[
\B_k=\B_{k,i}\lor\B_{k,j}.
\]
This completes the proof.
\end{proof}

\begin{lemma}\label{lem:threeclasses}
Each of \(\B_{4,1}\), \(\B_{4,2}\) and \(\B_{4,3}\) is a Cross
variety.
\end{lemma}

\begin{proof}
Analogously to Proposition~\ref{prop:seven}, for each \(\B_{4,i}\) we
shall give a finite list of interval semirings of bounded posets such
that the variety generated by every interval semiring of a bounded poset
in \(\B_{4,i}\) is equal to the variety generated by some of the
interval semirings of bounded posets in that list.

Let \(I(P)\) be an interval semiring in \(\B_{4,i}\). If
\(\htp(P)\leq 3\), then the seven templates of
Proposition~\ref{prop:seven} suffice. Next, it remains only to consider
the case \(\htp(P)=4\).

Let \(P\) have height four. Then the rigid points have ranks
\(0,1,2,3,4\), respectively. A nonrigid interior point has rank one or
two. A nonrigid point of rank two has no interior successor, and a
nonrigid point of rank one cannot lie below a rigid point of rank two,
since either comparison would put it on a full four-step chain. By
Lemma~\ref{lem:ipck}, every pair involving a nonrigid point is separated
by maps into an existing \(C_4\). It remains to separate rigid points of
equal rank.

\smallskip
\noindent\textbf{Case 1.} \(P\in\B_{4,1}\).
Then \(P\) has exactly one rigid point of rank one. For each of
\(L_3,U_3,E_3\) in Figure~\ref{fig:seven-posets}, adjoining a new bottom
gives three fixed induced targets of height four, denoted by
\(L_3',U_4,E_3'\), respectively; see Figure~\ref{fig:five-posets}.
Since every strictly order preserving map from \(P\) to one of \(L_3',U_4,E_3'\) fixes the
points of ranks \(0\) and \(1\), we may apply the height-three argument
of Proposition~\ref{prop:seven}. These maps preserve ranks, and hence
every pair of rigid points of equal rank in \(P\) can be separated by a
strictly order preserving map from \(P\) to one of the posets in
\(\{L_3',U_4,E_3'\}\). Extending such a map to all of \(P\) by sending
every nonrigid rank-one point to the target rank-one point and every
nonrigid rank-two point to a chosen target rank-three point preserves
strict order preservation, by the same rank restrictions as above. Thus a fixed list of
eleven semirings suffices.

\smallskip
\noindent\textbf{Case 2.} \(P\in\B_{4,3}\).
This case is the order dual of \(\B_{4,1}\): reversing multiplication
transports all constructions and the identity \(\tau_{4,1}\) to
\(\tau_{4,3}\) after renaming variables. In particular, every member of
\(\B_{4,1}\) is finitely based if and only if every member of
\(\B_{4,3}\) is finitely based, and the subvariety lattices of
\(\B_{4,1}\) and \(\B_{4,3}\) are isomorphic. Hence \(\B_{4,3}\) is also
a Cross variety.

\smallskip
\noindent\textbf{Case 3.} \(P\in\B_{4,2}\).
We separate \(P\) by strictly order preserving maps from \(P\) to \(L_4\) and \(U_4\); see
Figure~\ref{fig:five-posets}.
Suppose first that the rigid points of rank \(1\) in \(P\) are not
unique. For any two distinct rigid points \(a,b\) of rank \(1\) in
\(P\), define a strictly order preserving map \(f_{a,b}\colon P\to L_4\) by fixing \(\bot\)
and \(\top\), setting \(f(a)=a_1\) and \(f(b)=b_1\), and defining
\(f(p)=a_{\ell(\bot,p)}\) for every other interior point \(p\). Then
\(f_{a,b}\) is strictly order preserving and separates \(a\) from \(b\).

Similarly, for any two distinct rigid points \(a,b\) of rank \(3\) in
\(P\), define a strictly order preserving map \(g_{a,b}\colon P\to U_4\) by fixing \(\bot\)
and \(\top\), setting \(g(a)=a_3\) and \(g(b)=b_3\), and defining
\(g(p)=a_{\ell(\bot,p)}\) for every other interior point \(p\). Then
\(g_{a,b}\) is strictly order preserving and separates \(a\) from \(b\).

\smallskip
All selected targets are induced bounded subposets of the source.
Consequently every subvariety of \(\B_{4,1}\) or \(\B_{4,3}\) is
generated by a subset of a fixed list of eleven finite semirings; for
\(\B_{4,2}\) ten suffice. Hence each \(\B_{4,i}\) has only finitely many
subvarieties, and every subvariety of \(\B_{4,i}\) is finitely
generated. Since each \(\B_{4,i}\) is finitely based, finitely generated
and has only finitely many subvarieties, it follows that each
\(\B_{4,i}\) is a Cross variety. In particular, every subvariety of
\(\B_{4,i}\) is finitely based.
\end{proof}

\begin{figure}[htbp]
\centering
\begin{tikzpicture}[
  scale=0.7,
  elem/.style={circle, fill=black, inner sep=1.3pt},
  lab/.style={font=\footnotesize},
  edge/.style={thick}
]

\begin{scope}[xshift=0cm]
  \node[elem] (c4b)  at (0,0) {};
  \node[elem] (c4a1) at (0,1.2) {};
  \node[elem] (c4a2) at (0,2.4) {};
  \node[elem] (c4a3) at (0,3.6) {};
  \node[elem] (c4t)  at (0,4.8) {};
  \node[lab] at (0,-0.55) {\(\bot\)};
  \node[lab] at (0.50,1.2) {\(a_1\)};
  \node[lab] at (0.50,2.4) {\(a_2\)};
  \node[lab] at (0.50,3.6) {\(a_3\)};
  \node[lab] at (0,5.35)   {\(\top\)};
  \draw[edge] (c4b) -- (c4a1);
  \draw[edge] (c4a1) -- (c4a2);
  \draw[edge] (c4a2) -- (c4a3);
  \draw[edge] (c4a3) -- (c4t);
  \node[lab] at (0,-1.2) {\(C_4\)};
\end{scope}

\begin{scope}[xshift=3.4cm]
  \node[elem] (l4b)   at (0,0) {};
  \node[elem] (l4a1)  at (-1.0,1.2) {};
  \node[elem] (l4b1)  at (1.0,1.2)  {};
  \node[elem] (l4a2)  at (0,2.4)    {};
  \node[elem] (l4a3)  at (0,3.6)    {};
  \node[elem] (l4t)   at (0,4.8)    {};
  \node[lab] at (0,-0.55) {\(\bot\)};
  \node[lab] at (-1.60,1.2) {\(a_1\)};
  \node[lab] at (1.60,1.2)  {\(b_1\)};
  \node[lab] at (0.50,2.4)  {\(a_2\)};
  \node[lab] at (0.50,3.6)  {\(a_3\)};
  \node[lab] at (0,5.35)    {\(\top\)};
  \draw[edge] (l4b) -- (l4a1);
  \draw[edge] (l4b) -- (l4b1);
  \draw[edge] (l4a1) -- (l4a2);
  \draw[edge] (l4b1) -- (l4a2);
  \draw[edge] (l4a2) -- (l4a3);
  \draw[edge] (l4a3) -- (l4t);
  \node[lab] at (0,-1.2) {\(L_4\)};
\end{scope}

\begin{scope}[xshift=6.9cm]
  \node[elem] (u4b)   at (0,0) {};
  \node[elem] (u4a1)  at (0,1.2) {};
  \node[elem] (u4a2)  at (0,2.4) {};
  \node[elem] (u4a3)  at (-1.0,3.6) {};
  \node[elem] (u4b3)  at (1.0,3.6)  {};
  \node[elem] (u4t)   at (0,4.8)    {};
  \node[lab] at (0,-0.55) {\(\bot\)};
  \node[lab] at (0.50,1.2) {\(a_1\)};
  \node[lab] at (0.50,2.4) {\(a_2\)};
  \node[lab] at (-1.60,3.6) {\(a_3\)};
  \node[lab] at (1.60,3.6)  {\(b_3\)};
  \node[lab] at (0,5.35)    {\(\top\)};
  \draw[edge] (u4b) -- (u4a1);
  \draw[edge] (u4a1) -- (u4a2);
  \draw[edge] (u4a2) -- (u4a3);
  \draw[edge] (u4a2) -- (u4b3);
  \draw[edge] (u4a3) -- (u4t);
  \draw[edge] (u4b3) -- (u4t);
  \node[lab] at (0,-1.2) {\(U_4\)};
\end{scope}

\begin{scope}[xshift=10.4cm]
  \node[elem] (l4pb)   at (0,0) {};
  \node[elem] (l4pa1)  at (0,1.2) {};
  \node[elem] (l4pa2)  at (-1.0,2.4) {};
  \node[elem] (l4pb2)  at (1.0,2.4)  {};
  \node[elem] (l4pa3)  at (0,3.6)    {};
  \node[elem] (l4pt)   at (0,4.8)    {};
  \node[lab] at (0,-0.55) {\(\bot\)};
  \node[lab] at (0.50,1.2) {\(a_1\)};
  \node[lab] at (-1.60,2.4) {\(a_2\)};
  \node[lab] at (1.40,2.4)  {\(b_2\)};
  \node[lab] at (0.50,3.6)  {\(a_3\)};
  \node[lab] at (0,5.35)    {\(\top\)};
  \draw[edge] (l4pb) -- (l4pa1);
  \draw[edge] (l4pa1) -- (l4pa2);
  \draw[edge] (l4pa1) -- (l4pb2);
  \draw[edge] (l4pa2) -- (l4pa3);
  \draw[edge] (l4pb2) -- (l4pa3);
  \draw[edge] (l4pa3) -- (l4pt);
  \node[lab] at (0,-1.2) {\(L_3'\)};
\end{scope}

\begin{scope}[xshift=14.0cm]
  \node[elem] (e3pb)   at (0,0) {};
  \node[elem] (e3pa1)  at (0,1.2) {};
  \node[elem] (e3pa2)  at (-1.2,2.4) {};
  \node[elem] (e3pb2)  at (1.2,2.4)  {};
  \node[elem] (e3pa3)  at (-1.2,3.6) {};
  \node[elem] (e3pb3)  at (1.2,3.6)  {};
  \node[elem] (e3pt)   at (0,4.8)    {};
  \node[lab] at (0,-0.55) {\(\bot\)};
  \node[lab] at (0.55,1.2) {\(a_1\)};
  \node[lab] at (-1.75,2.4) {\(a_2\)};
  \node[lab] at (1.75,2.4)  {\(b_2\)};
  \node[lab] at (-1.75,3.6) {\(a_3\)};
  \node[lab] at (1.75,3.6)  {\(b_3\)};
  \node[lab] at (0,5.35)    {\(\top\)};
  \draw[edge] (e3pb) -- (e3pa1);
  \draw[edge] (e3pa1) -- (e3pa2);
  \draw[edge] (e3pa1) -- (e3pb2);
  \draw[edge] (e3pa2) -- (e3pa3);
  \draw[edge] (e3pb2) -- (e3pb3);
  \draw[edge] (e3pa3) -- (e3pt);
  \draw[edge] (e3pb3) -- (e3pt);
  \node[lab] at (0,-1.2) {\(E_3'\)};
\end{scope}

\end{tikzpicture}
\caption{Hasse diagrams of the five bounded posets.}
\label{fig:five-posets}
\end{figure}
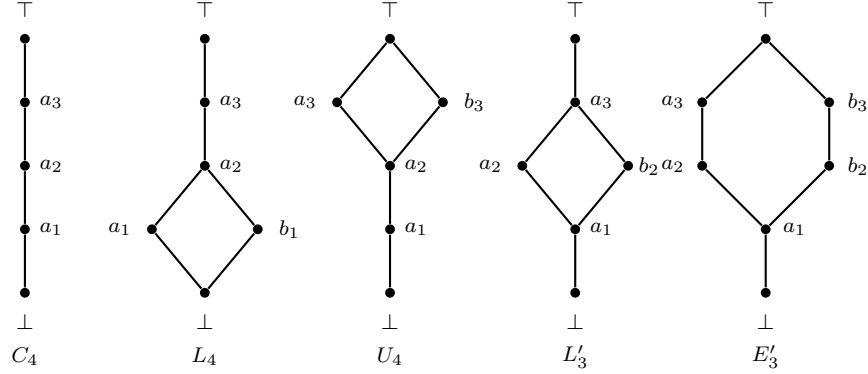

\begin{theorem}
The class of finite ai-semirings generating varieties with only finitely
many subvarieties is not closed under direct products. Moreover, a
direct product of such ai-semirings can even generate a variety with
continuum many subvarieties.
\end{theorem}

\begin{proof}
Consider the finite ai-semirings \(I(Q_{4,1})\) and \(I(Q_{4,2})\).
By Proposition~\ref{prop:bkiqki} and Lemma~\ref{lem:threeclasses}, they
generate the Cross varieties \(\B_{4,1}\) and \(\B_{4,2}\), respectively,
and hence each of these varieties has only finitely many subvarieties. On
the other hand, by Propositions~\ref{prop:k=4}
and~\ref{prop:bkbkibkj}, the variety \(\B_4\) is generated by their
direct product, and \(\B_4\) contains continuum many subvarieties.
\end{proof}

\begin{lemma}\label{lem:omitting}
Every subvariety of \(\V_\infty\) omitting \(I(E_4)\) is finitely
based.
\end{lemma}

\begin{proof}
Let \(\mathcal U\) be a subvariety of \(\V_\infty\) omitting
\(I(E_4)\). Since \(I(E_4)\in\V_5=\Var(I(C_5))\), the variety
\(\mathcal U\) omits \(I(C_5)\). This shows that every interval semiring
\(I(P)\) of a bounded poset \(P\) in \(\mathcal U\) has
\(\htp(P)\leq 4\). Hence \(\mathcal U\subseteq\B_4\).

Suppose that all three identities \(\tau_{4,1}\), \(\tau_{4,2}\) and
\(\tau_{4,3}\) fail in \(\mathcal U\). For each \(1\leq i\leq 3\), there
is a finite interval semiring \(I(P_i)\) of a bounded poset \(P_i\) such
that \(I(P_i)\in\mathcal U\) fails \(\tau_{4,i}\): an identity failure
first occurs in a finitely generated subalgebra and then survives in an
SI quotient. Its two four-step chains supply a strictly order preserving map \(E_4\to P_i\)
separating the two rank-\(i\) points. Also \(I(C_4)\in\mathcal U\), since
each witness contains a full chain; the map of both branches of \(E_4\)
onto \(C_4\) separates points of different ranks. These four maps jointly
separate all points of \(E_4\). By Lemma~\ref{lem:separation-hsp},
\(I(E_4)\in\mathcal U\), a contradiction.

Therefore \(\mathcal U\) satisfies some \(\tau_{4,i}\) and hence lies in
\(\B_{4,i}\). By Lemma~\ref{lem:threeclasses}, \(\mathcal U\) is
finitely based.
\end{proof}

\begin{theorem}\label{lem:classification}
Every nonfinitely based subvariety of \(\V_\infty\) contains
\(\Var(I(E_4))\). In particular, \(\Var(I(E_4))\) is the unique limit
subvariety of \(\V_\infty\).
\end{theorem}

\begin{proof}
By the contrapositive of Lemma~\ref{lem:omitting}, every nonfinitely
based subvariety of \(\V_\infty\) contains \(\Var(I(E_4))\). By
Theorem~\ref{thm:explicit}, \(\Var(I(E_4))\) itself is a limit variety.
Hence \(\Var(I(E_4))\) is the unique limit subvariety of
\(\V_\infty\).
\end{proof}

\end{document}